\documentclass[12pt]{article}

\usepackage{amsmath, amsthm, amssymb, accents,amscd}
\usepackage{enumerate}
\usepackage{pdflscape}
\usepackage{caption}

\usepackage{ifpdf}
\ifpdf
\usepackage[pdftex]{graphicx}
\else
\usepackage[dvips]{graphicx}
\fi
\usepackage{tikz}
 	 \usetikzlibrary{arrows,backgrounds}
\usepackage[all]{xy}
\usepackage{tikz-cd}
\usepackage{multicol}

\input xy
\xyoption{all} 

\usepackage[pdftex,plainpages=false,hypertexnames=false,pdfpagelabels]{hyperref}
\newcommand{\arxiv}[1]{\href{http://arxiv.org/abs/#1}{\tt arXiv:\nolinkurl{#1}}}
\newcommand{\arXiv}[1]{\href{http://arxiv.org/abs/#1}{\tt arXiv:\nolinkurl{#1}}}

\newcommand{\googlebooks}[1]{(preview at \href{http://books.google.com/books?id=#1}{google books})}

\usepackage{xcolor}
\definecolor{dark-red}{rgb}{0.7,0.25,0.25}
\definecolor{dark-blue}{rgb}{0.15,0.15,0.55}
\definecolor{medium-blue}{rgb}{0,0,.8}
\definecolor{DarkGreen}{RGB}{0,150,0}
\definecolor{rho}{named}{red}
\hypersetup{
   colorlinks, linkcolor={purple},
   citecolor={medium-blue}, urlcolor={medium-blue}
}

\usepackage{longtable}
\usepackage{fullpage}

\theoremstyle{plain}
\newtheorem{thm}{Theorem}[section]
\newtheorem*{thm*}{Theorem}
\newtheorem{thmalpha}{Theorem}

\newtheorem{cor}[thm]{Corollary}

\newtheorem*{cor*}{Corollary}

\newtheorem*{conj*}{Conjecture}
\newtheorem*{lem*}{Lemma}
\newtheorem{lem}[thm]{Lemma}
\newtheorem{prop}[thm]{Proposition}

\newtheorem*{quest*}{Question}
\newtheorem*{claim*}{Claim}

\theoremstyle{definition}

\newtheorem{defn}[thm]{Definition}

\theoremstyle{remark}

\newtheorem{rem}[thm]{Remark}

\DeclareMathOperator{\Aut}{Aut}

\DeclareMathOperator{\tDiff}
{\mathrm{D}\!\widetilde{\,i\hspace{1.5mm}}\hspace{-1.5mm}\mathrm{ff}} 
\DeclareMathOperator{\Diff}{Diff}
\DeclareMathOperator{\End}{End}

\DeclareMathOperator{\Ann}{Ann}

\DeclareMathOperator{\tAnn}{A\widetilde{\!\!\phantom{\imath}n\phantom{\imath}\!\!}n}

\DeclareMathOperator{\Hom}{Hom}
\DeclareMathOperator{\spann}{span}
\DeclareMathOperator{\id}{id}

\DeclareMathOperator{\Mob}{M\ddot{o}b}
\DeclareMathOperator{\Univ}{Univ}
\DeclareMathOperator{\tUniv}{U\widetilde{ni}v}

\DeclareMathOperator{\Vir}{Vir}

\newcommand{\relint}{\mathrm{Rel\text{-}Int}}

\newcommand{\comment}[1]{}

\newcommand{\be}{\begin{enumerate}[label=(\arabic*)]}
\newcommand{\ee}{\end{enumerate}}

\newcommand{\ip}[1]{\langle #1 \rangle}

\newcommand{\abs}[1]{\left| #1 \right|}

\def\semicolon{;}
\def\applytolist#1{
    \expandafter\def\csname multi#1\endcsname##1{
        \def\multiack{##1}\ifx\multiack\semicolon
            \def\next{\relax}
        \else
            \csname #1\endcsname{##1}
            \def\next{\csname multi#1\endcsname}
        \fi
        \next}
    \csname multi#1\endcsname}

\def\calc#1{\expandafter\def\csname c#1\endcsname{{\mathcal #1}}}
\applytolist{calc}QWERTYUIOPLKJHGFDSAZXCVBNM;
\def\bbc#1{\expandafter\def\csname bb#1\endcsname{{\mathbb #1}}}
\applytolist{bbc}QWERTYUIOPLKJHGFDSAZXCVBNM;
\def\bfc#1{\expandafter\def\csname bf#1\endcsname{{\mathbf #1}}}
\applytolist{bfc}QWERTYUIOPLKJHGFDSAZXCVBNM;
\def\sfc#1{\expandafter\def\csname s#1\endcsname{{\sf #1}}}
\applytolist{sfc}QWERTYUIOPLKJHGFDSAZXCVBNM;

\newcommand{\noshow}[1]{}
\newcommand{\MR}[1]{}

\makeatletter
\newcommand{\doublewidetilde}[1]{{%
  \mathpalette\double@widetilde{#1}%
}}
\newcommand{\double@widetilde}[2]{%
  \sbox\z@{$\m@th#1\widetilde{#2}$}%
  \ht\z@=.9\ht\z@
  \widetilde{\box\z@}%
}
\makeatother

\usetikzlibrary{shapes}
\usetikzlibrary{backgrounds}
\usetikzlibrary{decorations,decorations.pathreplacing,decorations.markings}
\usetikzlibrary{fit,calc,through}
\usetikzlibrary{external}
\tikzset{
	super thick/.style={line width=3pt}
}
\tikzstyle{shaded}=[fill=red!10!blue!20!gray!30!white]
\tikzstyle{unshaded}=[fill=white]
\tikzstyle{empty box}=[circle, draw, thick, fill=white, opaque, inner sep=2mm]
\tikzstyle{annular}=[scale=.7, inner sep=1mm, baseline]
\tikzstyle{rectangular}=[scale=.75, inner sep=1mm, baseline=-.1cm]
\tikzstyle{mid>}=[decoration={markings, mark=at position 0.5 with {\arrow{>}}}, postaction={decorate}]
\tikzstyle{mid<}=[decoration={markings, mark=at position 0.5 with {\arrow{<}}}, postaction={decorate}]
\tikzstyle{over}=[double, draw=white, super thick, double=]

\title{A genus zero functorial CFT associated to the vacuum sector of a conformal net}
\author{Andr\'e G. Henriques, James E. Tener}
\date{}

\begin{document}

\maketitle

\begin{abstract}
Starting from an arbitrary conformal net, we construct a genus zero functorial conformal field theory whose Hilbert space is the vacuum sector of the net.
Specifically, we construct an algebra for the operad of little discs and conformal embeddings with values in the category of Hilbert spaces and bounded linear maps.
We work with closed discs, and
our conformal embeddings explicitly allow the image of an incoming disc to overlap with the boundary of the outgoing disc.
As an application, we show that all conformal nets satisfy the trace class condition: if $L_0$ is the conformal Hamiltonian of the net, then the operators $r^{L_0}$ are trace class whenever $0 \le r < 1$.
In particular, the $L_0$-eigenspaces of a conformal net are automatically finite-dimensional.
\end{abstract}

\tableofcontents

\section{Introduction}\label{sec: introduction}

\emph{Conformal nets} are a formalisation of unitary 2d chiral conformal field theories in the spirit of Haag-Kastler algebraic quantum field theory.
The key data of a conformal net is a family of von Neumann algebras $\cA(I)$ indexed by intervals $I$ of the unit circle $S^1$. 
These von Neumann algebras act on a common Hilbert space $H_0$ called the vacuum sector of the chiral CFT, and transform appropriately under a projective unitary positive energy representation of the group $\Diff(S^1)$.

The first goal of this article is to construct, for every conformal net, the data of a genus zero functorial field theory.
This construction provides bounded linear maps associated to every configuration of discs $D_1, \ldots, D_n$ disjointly embedded in a larger disc $D$. We think of such a bounded linear map as the evolution operator associated to the Riemann surface $D \setminus (\mathring D_1 \cup \cdots \cup \mathring D_n)$, which is a genus zero cobordism between $n$ circles and one circle.

\subsubsection*{Statement of main results}

We now outline more precisely the structure that we will construct.

\begin{defn}\label{def: multidisc}
By a \emph{disc}, we shall mean a Riemann surface with boundary which is isomorphic to the closed unit disc in the complex plane (here and throughout the article, holomorphic maps of Riemann surfaces are taken to be smooth up to the boundary), and a \emph{multidisc} is a finite disjoint union of discs.
A \emph{multidisc embedding} is a family of discs $D_1, \ldots, D_n, D$ along with a holomorphic (smooth up to the boundary) embedding of discs $\mathfrak d:D_1 \sqcup \cdots \sqcup D_n \to D$.
For a nonnegative integer $n$, let $\mathfrak D(n)$ denote the collection of all multidisc embeddings of $n$ discs $D_1, \ldots, D_n$ into a disc $D$, and let $\mathfrak D = \bigcup_{n\in\bbN} \mathfrak D(n)$.

Given a finite dimensional complex manifold $M$ (without boundary), a \emph{holomorphic family of multidisc embeddings} parametrized by $M$ is a holomorphic map $\mathfrak d: M \times (D_1 \sqcup \cdots \sqcup D_n) \to D$ such that for each $m \in M$ the induced map $\mathfrak d_m := \mathfrak d(m, -)$ is a multidisc embedding $D_1 \sqcup \cdots \sqcup D_n \to D$.
\end{defn}

The collections of all discs and all multidisc embeddings respectively form the colours and the operations of an operad (also known as the objects and morphisms of a multicategory).

\begin{defn}\label{defn: operad of conformal discs}
We denote by $\mathfrak D$ the operad whose colours are conformal discs and whose operations are multidisc embeddings, and call it the \emph{operad of conformal discs}.
\end{defn}

Note that the identity map $D\to D$ is an allowed operation in $\mathfrak D$ (the operad of conformal discs is a unital operad).
More generally, we do not impose any restrictions on the boundary intersections $\mathfrak d(\partial D_j) \cap \partial D$.

Starting from a conformal net, we will construct an algebra for the operad of conformal discs valued in the category of Hilbert spaces and bounded linear maps.
That is, for each disc $D$ we have a Hilbert space $H_0(D)$, the vacuum sector associated to the disc $D$, and for each multidisc embedding $\mathfrak d : D_1\sqcup \cdots \sqcup D_n \to D$ we have a bounded linear map
\[
Y_{\mathfrak d}:H_0(D_1) \otimes \cdots \otimes H_0(D_n) \to H_0(D).
\]
These maps satisfy the operadic composition law
\[
Y_{\mathfrak d' \circ_i \mathfrak d} = Y_{\mathfrak d'} \circ_i Y_{\mathfrak d} := Y_{\mathfrak d'} \circ (\mathrm{id} \otimes \cdots \otimes Y_{\mathfrak d} \otimes \cdots \otimes \mathrm{id})
\]
whenever the composition $\mathfrak d' \circ_i \mathfrak d$ is defined.
{Moreover, when $\mathfrak d_m$ is a holomorphic family of multidisc embeddings $D_1\sqcup \cdots \sqcup D_n \to D$, the operators $Y_{\mathfrak d_m}$ depend holomorphically on $m$, for the norm topology on $B\big(H_0(D_1) \otimes \cdots \otimes H_0(D_n),\,H_0(D)\big)$.}

We now expand on the particular construction which will follow in the body of the article.
Let $\cA$ be a conformal net.
We begin by recalling \cite{BartelsDouglasHenriques15} that there is a \emph{coordinate-free conformal net} associated to $\cA$, which we denote with the same symbol.
The coordinate-free conformal net assigns to every abstract (i.e. not necessarily embedded) interval $I$ an abstract von Neumann algebra (i.e. a $W^*$-algebra).
This construction is functorial in the interval $I$, and given an isomorphism $f: I_1 \rightarrow I_2$ between intervals we write $f_*:\cA(I_1) \to \cA(I_2)$ for the corresponding isomorphism of von Neumann algebras.

For each disc $D$ there is also an associated Hilbert space $H_0(D)$ along with a distinguished vacuum vector $\Omega_D \in H_0(D)$, and
for each interval $I \subset \partial D$ there is a representation of $\cA(I)$ on $H_0(D)$.
In the special case when $D$ is the unit disc $\bbD \subset \bbC$, we recover the original vacuum Hilbert space and vacuum vector of the conformal net.
{If $\varphi:\partial D_1 \to \partial D_2$ is an orientation-preserving diffeomorphism, there is an associated unitary operator $U_{\varphi}:H_0(D_1) \to H_0(D_2)$, well-defined up to phase.
Moreover, when $\varphi$ extends to a biholomorphic map $D_1 \cong D_2$, the phase ambiguity can be lifted, and there is a canonical choice of $U_\varphi$ characterised by the requirement that $U_{\varphi}\Omega_{D_1} = \Omega_{D_2}$.}

In \cite{HenriquesTenerWorms}, we constructed vectors 
\[
|x_1 \cdots x_n\rangle_{D} \in H_0(D)
\]
in the vacuum Hilbert spaces $H_0(D)$ corresponding to configurations of disjoint extendable\footnote{An interval $I$ in a disc $D$ is called extendable if there exists an interval $I_+ \subset D$ containing $I$ in its interior, and if the orientations of $I$ and $\partial D$ agree on their intersection.} intervals $I_1, \ldots, I_n \subset D$, each labelled by an element $x_j \in \cA(I_j)$ of the corresponding von Neumann algebra\footnote{We call these `worm insertions,' analogous to the more standard notion of `point insertions' of finite-energy vectors.}.
For brevity, we shall sometimes denote a collection $I_1, \ldots, I_n$ of intervals by $\underline{I}$, and a tuple of algebra elements by $\underline x \in \cA(\underline I)$, where $\cA(\underline I):= \cA(I_1) \times \cdots \times \cA(I_n)$.
In this case the corresponding insertion is denoted $| \underline x \rangle_D$.
The empty insertion $| \,\, \rangle_D$ agrees with the vacuum vector $\Omega_D$.

We now describe the functorial chiral CFT associated to a conformal net $\cA$.
To a disc $D$, we associate the vacuum Hilbert space $H_0(D)$.
The linear map $\bbC \to H_0(D)$ corresponding to the embedding $\emptyset \to D$ in $\mathfrak D(0)$ sends $1 \mapsto \Omega_D$.
The general construction of an algebra for the operad of conformal discs is given as follows (Theorem~\ref{thm: Yd exists} {and Proposition~\ref{prop: holomorphic dependence}} in the body of the article).
\begin{thmalpha}\label{thmalpha: genus zero Segal CFT}
    For each multidisc embedding $\mathfrak d:D_1 \sqcup \cdots \sqcup D_n \to D \in \mathfrak D$, there exists a unique bounded linear map 
    \[
    Y_{\mathfrak d}:H_0(D_1) \otimes \cdots \otimes H_0(D_n) \to H_0(D)
    \]
    characterised by the property that for every set $\underline I_1,\ldots,\underline I_n$ of collections of disjoint extendable intervals, with $\underline I_j$ in $D_j$, and every $\underline x_j \in \cA(\underline I_j)$, we have 
    \[
Y_{\mathfrak d} \big( |\underline x_1\rangle_{D_1} \otimes \cdots \otimes |\underline x_n\rangle_{D_n} \big) = | \mathfrak d_*(\underline x_1) \cdots \mathfrak d_*(\underline x_n) \rangle_D.
\]
These maps are compatible with the composition of multidisc embeddings, i.e., they form an algebra for the operad of conformal discs.

Furthermore, if $\mathfrak d_m:D_1 \sqcup \cdots \sqcup D_n \to D$ is a holomorphic family of multidisc embeddings parametrised by $m\in M$, for some complex manifold $M$ (see Definition~\ref{def: multidisc}), then the operators $Y_{\mathfrak d_m}$ depend holomorphically on the parameter $m$, for the norm topology on $B\big(H_0(D_1) \otimes \cdots \otimes H_0(D_n), H_0(D)\big)$.
\end{thmalpha}
Of particular emphasis in the above theorem is that the operators constructed are bounded maps between Hilbert spaces, where the domain is the Hilbert space tensor product.
This solves an analytic problem first formulated in \cite[\S 4.4.2]{Henriques14}, and provides a crucial missing piece for later work in which we construct full CFTs from conformal nets.

Related structures were constructed in the work of Huang and Bruegmann \cite{HuangFunctionalI,BruegmannFactorizationAlgebra}; the key difference in our result is the presence of Hilbert spaces and bounded linear maps, as opposed to more general spaces (locally convex topological vector spaces and bornological spaces, respectively).

\subsubsection*{Our main result in terms of cobordisms}

Given a multidisc embedding $\mathfrak d:D_1 \sqcup \cdots \sqcup D_n \to D$, the associated \emph{complex cobordism} is the topological space  $\cR:=D \setminus (\mathfrak d(\mathring D_1) \cup \cdots \cup \mathfrak d(\mathring D_n))$, where $\mathring D_i$ denotes the interior of $D_i$. It is a locally ringed space when equipped with the sheaf of $\bbC$-valued functions that are continuous, holomorphic in the interior, and smooth on the various boundary circles.
Given such a complex cobordism $\cR$ along with boundary parametrisations $\varphi_j:\partial \bbD \to \partial D_j$ and $\varphi_0:\partial\bbD \to \partial D$, we show in Corollary~\ref{cor: YR well defined} that the bounded operator
\[
Y_{\cR}:=U_{\varphi_0}^* \circ Y_{\mathfrak d} \circ (U_{\varphi_1} \otimes \cdots \otimes U_{\varphi_n}) : H_0^{\otimes n} \to H_0
\]
only depends, up to a non-zero scalar, on the cobordism $\cR$ and its boundary parametrizations, and not on the presentation of $\cR$ in terms of discs.
We note that this result does not follow formally from Theorem~\ref{thmalpha: genus zero Segal CFT}, i.e., it does not follow formally from the fact that we have an algebra over the operad of conformal discs.

In Corollary~\ref{cor: YR compatible with decomposition}, we furthermore show that the maps $Y_{\cR}$ are compatible with composition of cobordisms, in the sense that whenever a complex cobordism $\cR$ can be decomposed as $\cR=\cR_1\cup_{S^1} \cR_2$ then we have 
\[
Y_{\cR}=Y_{\cR_1} \circ (\mathrm{id} \otimes \cdots \otimes Y_{\cR_2} \otimes \cdots \otimes \mathrm{id}),
\]
up to scalar.
This result relates to Huang's early work \cite{HuangFunctionalII} in roughly the same way that Theorem~\ref{thmalpha: genus zero Segal CFT} relates to \cite{HuangFunctionalI}, our main contribution being an upgrade to Hilbert spaces and bounded linear maps from the earlier setup using locally convex topological vector spaces.
These operators $Y_\cR$ also subsume the bounded operators assigned to certain surfaces $\cR$ for certain conformal nets $\cA$ in \cite{GRACFT1}.

\subsubsection*{The trace class condition}

As an application of Theorem~\ref{thmalpha: genus zero Segal CFT}, we obtain the following structural result for conformal nets  (Theorem~\ref{thm: rL0 trace class} in the body of the article):

\begin{thmalpha}\label{thmalpha: trace class}
    For every conformal net $\cA$, the operator $r^{L_0}:H_0 \to H_0$ is trace class whenever $0 \le r < 1$. In particular, the eigenspaces of $L_0$ are finite-dimensional.

    Here, $L_0$ denotes the conformal Hamiltonian, i.e., the self-adjoint generator of the one-parameter group of rotations of $S^1$.
\end{thmalpha}
Note that, in our definition of conformal net, the eigenspaces of $L_0$ are not assumed to be a priori finite-dimensional. It follows from Theorem~\ref{thmalpha: trace class} that they must always be so.

The condition that the operators $r^{L_0}$ be trace class (called the \emph{trace class condition}) and related nuclearity properties have been studied in the context of conformal nets since the early history of the theory.
One early application in these articles was the use of the trace class condition to establish the split property of conformal nets (and thereby hyperfiniteness of local algebras) \cite{BuchholzDantoniLongoNuclearMapsII, GabbianiFrohlich93}, before the split property was shown to be automatic by Morinelli-Tanimoto-Weiner \cite{MorinelliTanimotoWeiner18}.
The trace class condition was also shown to imply several other nuclearity conditions \cite{BuchholzDantoniLongoNuclearityAndThermalStates}.
Another consequence of the trace class condition is the convergence of the vacuum character 
\[
\chi_0(q) := \operatorname{tr}(q^{L_0-c/24}) = \sum_{n=0}^\infty \dim \big(\operatorname{ker}(L_0-n)\big) q^{n-c/24}
\]
in the unit disc (where $c$ is the central charge of the theory).

\subsubsection*{Connections to vertex operator algebras}

Unitary vertex operator algebras (VOAs) are a second axiomatization of unitary 2d chiral conformal field theories.
The systematic comparison between unitary VOAs and conformal nets was initiated in \cite{CKLW18}, and the authors conjectured that there is a bijective correspondence between unitary VOAs and conformal nets.
In \cite{HenriquesTenerWorms}, we established one direction of this conjectured correspondence, by showing that there is a unitary vertex operator algebra associated to every conformal net.
In this previous work, we considered only conformal nets whose associated $L_0$-eigenspaces were finite-dimensional, a condition that we now know to be automatic as a result of Theorem~\ref{thmalpha: trace class}.
The finite-dimensionality of eigenspaces was used in \cite[\S 6]{HenriquesTenerWorms} to construct vectors associated to insertions of finite-energy vectors at points in discs, analogous to the construction of vectors associated to `worm' insertions of conformal net algebra elements labelling intervals\footnote{Crucially, the construction of vectors associated to worm insertions in \cite[\S5]{HenriquesTenerWorms} does not use the finite-dimensionality of the $L_0$-eigenspaces.}.
In Appendix~\ref{sec: point insertions}, we show that the operators $Y_{\mathfrak d}$ described in Theorem~\ref{thmalpha: genus zero Segal CFT} are compatible with point insertions. 


\subsection*{Acknowledgements}

The second author was supported by ARC Discovery Project DP200100067. For the purpose of Open Access, the authors have applied a CC BY public copyright licence to any Author Accepted Manuscript (AAM) version arising from this submission.

\section{Background}

\subsection{The semigroup of annuli}\label{sec: annuli}

The \emph{Virasoro Fr\'echet Lie algebra}
\begin{align*}
&\,\,\,\,\Vir\,\,\, :=\,\, \textstyle \bigg\{\sum_{n\in \bbZ} a_n L_n+ kC \,\bigg|\, 
\begin{aligned}
&a_n, k \in \bbC\\[-1mm]
&\text{$a_n$ is rapidly decreasing as $|n|\to\infty$}
\end{aligned}
\bigg\}
\\
&[L_m,L_n] = (m-n)L_{n+m} + \frac{C}{12} (m^3-m) \delta_{n+m,0}
\end{align*}
is the universal central extension of the complexification of the Lie algebra of $\Diff(S^1)$, the group of orientation-preserving diffeomorphisms of the circle $S^1$. It contains the \emph{algebraic Virasoro algebra} 
$\{ \sum a_n L_n+ kC \,|\, a_n$ eventually zero$\}$
as a dense subalgebra.
Using the energy bounds of Goodman and Wallach \cite{GoWa85}, one can show that every positive energy unitary representation $(V,\pi)$ of the algebraic Virasoro algebra extends to a representation of the Virasoro Fr\'echet Lie algebra by unbounded operators on the Hilbert space completion $H$ of $V$ (\cite[Prop. 4.16]{HenriquesTenerWorms}). 
The operator $X \in \Vir$ acts on $H$ as the closure of $\pi(X)$, regarded as an unbounded operator on $H$ with domain $V$. By abuse of notation, we denote this closure by $\pi(X)$ as well.

The subalgebra
$\Vir_\bbR := \mathrm{Span}_\bbR(\{L_n - L_{-n}, iL_n + iL_{-n}\}, iC)\subset \Vir$ is the Lie algebra of a certain infinite-dimensional group known as the Bott-Virasoro group, which is a central extension of $\Diff(S^1)$. 
But the Lie algebra $\Vir$ is not itself the Lie algebra of any Fr\'echet Lie group (cf. \cite{SegalDef,Neretin90}).
There is however a semigroup that can reasonably be said to integrate the Virasoro Fr\'echet Lie algebra. This semigroup is a central extension of the 
\emph{semigroup of annuli}, described in Definition~\ref{def: def annuli} below.

Given a smooth Jordan curve $\gamma:S^1\to \bbC$ with winding number $1$, let $\mathrm{Int}(\gamma)\subset \bbC$ denote the bounded component of $\bbC\setminus \gamma(S^1)$.

\begin{defn}\label{def: def annuli}
The semigroup of annuli, denoted $\Ann$, is the set of equivalence classes of pairs of smooth Jordan curves $\varphi_{in},\varphi_{out}:S^1\to \bbC$ satisfying $\mathrm{Int}(\varphi_{in})\subseteq \mathrm{Int}(\varphi_{out})$. Two pairs $(\varphi_{in},\varphi_{out})$ and $(\varphi'_{in},\varphi'_{out})$ are declared equivalent if there exists a $C^\infty$ diffeomorphism $f:\bbC\to \bbC$ such that $f\circ \varphi_{in/out}=\varphi'_{in/out}$ and $f|_{\mathrm{Int}(\varphi_{out}) \setminus \overline{\mathrm{Int}(\varphi_{in})}}$ is holomorphic.
\end{defn}

Given a pair of smooth Jordan curves $\varphi_{in/out}$ as above, we identify it with the annulus $A = \overline{\mathrm{Int}(\varphi_{out})} \setminus \mathrm{Int}(\varphi_{in})$, and we refer to the image of $\varphi_{in}$ (resp. $\varphi_{out}$) as the incoming (resp. outgoing) boundary of $A$. We write $\partial_{in}A$ for the incoming boundary of $A$, and $\partial_{out}A$ for the outgoing boundary.
In \cite[\S3]{HenriquesTener24ax}, we showed that $\Ann$ is a semigroup under the operation $(A_1,A_2) \mapsto A_1 \cup A_2$ of \emph{conformal welding}, which identifies the outgoing boundary of $A_2$ with the incoming boundary of $A_1$ along their respective parametrizations.

There is a natural embedding $\Diff(S^1) \hookrightarrow \Ann$ which sends
a diffeomorphism $\varphi \in \Diff(S^1)$ to the completely thin annulus $S^1$ with outgoing boundary parametrized by the identity map, and incoming boundary parametrized by $\varphi$.

The semigroup of annuli is equipped with an involution $\dagger:\Ann \to \Ann$ which sends an annulus $A$ to the same annulus but equipped with the opposite complex structure.
The boundary parametrisations remain the same, but the incoming and outgoing boundaries are interchanged. One readily checks that
this involution satisfies $(A \cup B)^\dagger = B^\dagger \cup A^\dagger$.

In \cite{HenriquesTener24ax}, we constructed a central extension
\begin{equation}\label{eq: anntc}
0\to \bbC^\times \times \bbZ \to \tAnn_c \to \Ann \to 0
\end{equation}
of the semigroup of annuli, depending on a parameter $c\in\bbR$ called \emph{central charge}.
This semigroup carries an involution $\dagger:\tAnn_c\to \tAnn_c$ covering the involution $\dagger$ on $\Ann$.
In \cite{HenriquesTenerIntegratingax}, we showed that every positive energy unitary representation $(V,\pi)$ of $\Vir$ with central charge $c$ (one where $C$ acts by the scalar $c$) integrates to a holomorphic $*$-representation of $\tAnn_c$ (denoted again $\pi$) by bounded operators on the Hilbert space completion of $V$.
The central extension $\tAnn_c$ of $\Ann$ 
restricts to a central extension of $\Diff(S^1) \subset \Ann$ by $U(1)\times \bbZ$ which we denote by $\tDiff_c(S^1)$, and a $*$-representation of $\tAnn_c$ restricts to a unitary representation of $\tDiff_c(S^1)$.
Given an annulus $A \in \Ann$ and a lift $\underline A \in \tAnn_c$, we shall typically denote the corresponding Hilbert space operator again by $\underline A$.

The semigroup $\Ann$ admits an important sub-semigroup which behaves like the Borel subgroup of an algebraic group:

\begin{defn}\label{def: univ}
The \emph{semigroup of univalent maps} $\Univ$ is the set of holomorphic embeddings $f : \bbD \to \bbD$ (with derivative everywhere non-zero, including on $\partial \bbD$).
The inclusion $\Univ \hookrightarrow \Ann$ sends a univalent map $f$ to the annulus
$A_f := \bbD \setminus f(\mathring{\bbD})$, where the boundary parametrisations of $A_f$ are provided by $\varphi_{in} = f$, and $\varphi_{out} = \id$, respectively.
\end{defn}

Let $\tUniv$ be the universal cover of $\Univ$.
As explained in \cite[\S4.2]{HenriquesTenerWorms}, the inclusion $\Univ \hookrightarrow \Ann$ uniquely lifts to an inclusion $\tUniv \hookrightarrow \tAnn_c$.
The image of the central $\bbZ\subset \tUniv$ (generated by $2\pi$-rotation) under that embedding singles out a canonical subgroup isomorphic to $\bbZ$ inside the center $\bbC^\times\times \bbZ$ of $\tAnn_c$.
We shall denote by $\Ann_c$ the quotient of $\tAnn_c$ by this subgroup. It fits into a central extension:
\begin{equation}\label{eq: 1st appearance of Annc}
0\to \bbC^\times \to \Ann_c \to \Ann \to 0.
\end{equation}
Similarly, the quotient of $\tDiff_c(S^1)\subset \tAnn_c$ by $\bbZ$ is denoted $\Diff_c(S^1)$; it fits into a central extension:
\[
0\to U(1) \to \Diff_c(S^1) \to \Diff(S^1) \to 0.
\]

In any given representation,
the generator of the central $\bbZ\subset \tUniv\subset \tAnn_c$ acts by $e^{2\pi iL_0}$.
So if $V$ is a Virasoro representation in which $L_0$ acts with integer spectrum, then the action  of $\tAnn_c$ on $H_V$ descends to an action of $\Ann_c$.
This will be the case for all representations considered in this article.
Using the inclusion $\tUniv \hookrightarrow \tAnn_c$, univalent annuli $A_f \in \Univ$ have a canonical lift to $\Ann_c$, which we again denote $A_f$.

\subsection{Conformal nets}\label{sec: conformal nets}

As before, let $\Diff(S^1)$ denote the group of orientation-preserving diffeomorphisms of $S^1$, and let $\Mob:=\mathit{PSU}(1,1)$ be its subgroup of M\"obius transformations.

\begin{defn}\label{def: definition of conformal net}
A \emph{conformal net} $\cA$ consists of:
\begin{itemize}
\item
A Hilbert space $H_0$ called the \emph{vacuum sector}, with a unit vector $\Omega \in H_0$ called the \emph{vacuum vector}.
\item An assignment of a von Neumann algebra $\cA(I)\subset B(H_0)$ to every interval\footnote{An \emph{interval} $I \subset S^1$ is a proper, nonempty, connected, closed subset of $S^1$} $I \subset S^1$.
\item
A strongly continuous representation $\Mob\to U(H_0):\varphi\mapsto U_\varphi$, along with a strongly continuous projective representation $\Diff(S^1)\to PU(H_0):\varphi\mapsto [U_\varphi]$ extending it.
\end{itemize}
These are required to satisfy:
\begin{itemize}
\item (Isotony) Whenever $I \subset J$, we have $\cA(I) \subset \cA(J)$.
\item (Locality) Whenever $\mathring I \cap \mathring J = \emptyset$, we have $[\cA(I),\cA(J)]=0$.
\item (Diffeomorphism covariance) 
The map $I \mapsto \cA(I)$ is $\Diff(S^1)$-equivariant, meaning $U_{\varphi}\cA(I)U_{\varphi}^* = \cA(\varphi(I))$ for all $\varphi \in \Diff(S^1)$. Moreover, if $\varphi\in \Diff(S^1)$ fixes $I$ pointwise then $\varphi$ acts as the identity on $\cA(I)$.
\item (Positivity of energy) The infinitesimal generator $L_0$ of the rotation subgroup of $\Mob$ has nonnegative spectrum.
\item (Vacuum) The vacuum vector $\Omega$ is invariant under $\Mob$, and it is cyclic for the joint actions of the algebras $\cA(I)$.
\item (Irreducibility) The vacuum vector is the unique vector fixed by $\Mob$, up to scalar.
\end{itemize}
\end{defn}

We do not assume, a priori, that the eigenspaces $V(n):=\ker(L_0-n)$ are finite-dimensional, although the irreducibility axiom implies $V(0)=\bbC \Omega$. 
We will show in Theorem~\ref{thm: rL0 trace class} that the finite-dimensionality of the $L_0$-eigenspaces follows automatically from the other axioms.

A key structural property of conformal nets is the \emph{split property}: if $I$ and $J$ are disjoint (closed) intervals of $S^1$, then the map $\cA(I) \odot \cA(J) \to B(H_0)$ given by $x \otimes y \mapsto xy$ extends to an isomorphism of von Neumann algebras $\cA(I) \otimes \cA(J) \cong \cA(I) \vee \cA(J)$. Here $\cA(I) \odot \cA(J)$ is the algebraic tensor product, $\cA(I) \otimes \cA(J)$ is the spatial tensor product (i.e. the von Neumann subalgebra of $B(H_0 \otimes H_0)$ generated by $\cA(I) \odot \cA(J)$), and $\cA(I) \vee \cA(J) \subset B(H_0)$ is the von Neumann subalgebra generated by $\cA(I)$ and $\cA(J)$.
This property was shown to hold for all conformal nets in \cite{MorinelliTanimotoWeiner18}, and going forward we will often suppress this isomorphism and write $\cA(I) \otimes \cA(J)$ instead of $\cA(I) \vee \cA(J)$.

We also note the \emph{Reeh-Schlieder theorem} for $\cA$, which asserts that the vacuum vector $\Omega$ is cyclic and separating for each von Neumann algebra $\cA(I)$.

Given a conformal net $\cA$, the projective representation of $\Diff(S^1)$ on $H_0$ differentiates to a positive energy unitary representation of the Virasoro algebra, with some central charge $c$, on the finite energy vectors $V = \bigoplus_{n\ge 0} V(n)$  (see \cite[Prop. 4.16]{HenriquesTenerWorms} or \cite[\S3.2]{CKLW18}). 
Thus, as described in Section~\ref{sec: annuli} (see also \cite{HenriquesTenerIntegratingax} and \cite[\S4.3]{HenriquesTenerWorms}), there exists a holomorphic representation of $\Ann_c$ on $H_0$ by bounded operators which lifts and extends the original representation of $\Diff(S^1)$.

Given a conformal net $\cA$, there is a corresponding \emph{coordinate-free conformal net} \cite[Prop. 4.3]{BartelsDouglasHenriques15}\footnote{The conformal nets considered in \cite{BartelsDouglasHenriques15} were assumed to satisfy the strong additivity property. That assumption, however, is not needed for any the constructions relevant to the present paper.}, denoted by the same letter. It assigns a von Neumann algebra $\cA(I)$ to every abstract interval $I$ (no longer a subset of $S^1$), and a Hilbert space $H_0(D)$ to every abstract disc $D$.
The elements of $\cA(I)$  are equivalence classes of pairs $(f,a)$, where $f:I\to S^1$ is an embedding, and $a\in \cA(f(I))$ is an element, where two pairs $(f,a)$ and $(f',a')$ are declared equivalent if there exists a diffeomorphism $\varphi\in \Diff(S^1)$ such that $\varphi\circ f=f'$ and $U_\varphi a U_\varphi^*=a'$.
Given a diffeomorphism $\varphi:I_1 \to I_2$, we have an associated isomorphism $\varphi_*:\cA(I_1) \to \cA(I_2)$ given by $[(f,a)]\mapsto [(f\circ \varphi^{-1},a)]$.
Moreover, if $I$ is an interval in $S^1$, then the two possible meanings of $\cA(I)$ are canonically identified via the bijection $a \leftrightarrow [(\id,a)]$.

Similarly, given an abstract disc $D$, the coordinate-free conformal net assigns to it the Hilbert space $H_0(D)$ whose elements are equivalence classes $(f,\xi)$ where $f:\bbD\to D$ is a biholomorphic map and $\xi \in H_0$, where two pairs $(f,\xi)$ and $(f',\xi')$ are equivalent if $U_{{f'}^{-1}\circ f}(\xi)=\xi'$. For $I\subset \partial D$, the algebra $\cA(I)$ acts on $H_0(D)$ by $x {\cdot} [(f,\xi)]:=[(f,f^{-1}_*(x)\xi)]$.
Note that there is a canonical identification $H_0 \cong H_0(\bbD)$ given by $\xi \mapsto (\mathrm{id},\xi)$, compatible with the actions of algebras $\cA(I)$.

The assignment $I\mapsto \cA(I)$ is not only functorial with respect to diffeomorphisms, but also with respect to embeddings (both orientation-preserving and orientation-reversing).
Given an orientation-preserving embedding $\varphi:I_1\to I_2$, we have an associated $*$-algebra homomorphism $\cA(I_1) \to \cA(I_2)$.
And given an orientation-reversing embedding $\varphi:I_1\to I_2$, we have an associated $\bbC$-linear $*$-algebra anti-homomorphism\footnote{A different convention would assign to an orientation-reversing embedding $\varphi:I_1\to I_2$ an anti-linear algebra homomorphism $\cA(I_1) \to \cA(I_2)$. The two conventions differ by an application of $*$.} $\varphi_*:\cA(I_1) \to \cA(I_2)$. Similarly, 
given a holomorphic isomorphism $D_1\to D_2$ between two discs we have a corresponding unitary $H_0(D_1) \to H_0(D_2)$, and
given an antiholomorphic isomorphism $D_1\to D_2$ we have an antiunitary $H_0(D_1) \to H_0(D_2)$.

If $D_1$ and $D_2$ are discs and $\varphi:\partial D_1 \to \partial D_2$ is an orientation-preserving diffeomorphism (not assumed to extend to a holomorphic map $D_1 \to D_2$), then there also exists a unique up to phase unitary 
\begin{equation}\label{eqn: unitary implementing diffeomorphism}
    U_\varphi:H_0(D_1) \to H_0(D_2)
\end{equation}
which \emph{implements} $\varphi$, which is to say that for every interval $I \subset \partial D_1$ and every algebra element $a \in \cA(I)$ we have $U_\varphi a U_\varphi^* = \varphi_*(a)$.
When $D_1 = D_2 = \bbD$, implementing unitaries are given by the projective representation $\Diff(S^1) \to PU(H_0)$.
And when $\varphi$ extends to a biholomorphic isomorphism $\varphi:D_1 \to D_2$, then the phase ambiguity can be lifted and we recover the unitary mentioned in the previous paragraph.

Let $A$ be an annulus without boundary parametrizations, written as $A= D_{out}\setminus \mathring D_{in}$ for some discs $D_{in}\subset D_{out}$. In \cite[Section~4.4]{HenriquesTenerWorms}, we introduced a bounded operator
\begin{equation}\label{eqY}
Y_A:H_0(D_{in})\to H_0(D_{out})
\end{equation}
defined in the following way.
Choose biholomorphic maps $f_{in} : \bbD \to D_{in}$ and $f_{out} : \bbD \to D_{out}$.
Observe that $f_{out}^{-1}\circ f_{in}\in \Univ$, and let
\begin{equation}\label{eqn: def YA}
Y_A\big([(f_{in}, \xi)]\big) := \big[\big(f_{out}, \pi(f_{out}^{-1}\circ f_{in}) \xi\big)\big],
\end{equation}
where $\pi:\Univ\subset \Ann_c\to B(H_0)$ is the action of $\Univ$ discussed in Section~\ref{sec: annuli}.
We recall that, by \cite[Lem 4.15]{HenriquesTenerWorms}, the map $Y_A$ is injective with dense image.

We finish by noting that
if $f:\bbD \to \bbD$ is a univalent map and $A_f = \bbD \setminus f(\mathring \bbD)$, then the operators \eqref{eqn: unitary implementing diffeomorphism} and \eqref{eqY} are compatible in the sense that $Y_{A_f} \circ U_f = \pi(f)$.

\subsection{Worm insertions}

In this section we will summarise results from \cite[\S5]{HenriquesTenerWorms} in which we construct vectors (resp. operators) associated to discs (resp. annuli) decorated by intervals $I_1\ldots I_n$ which are further labelled by elements $x_1\ldots x_n$ of the conformal net local algebras (`worm insertions').
\begin{equation}
\label{eq: skdjnbksdfksb}
\left[\,\,\begin{tikzpicture}[scale=0.4,baseline]
  \filldraw[fill=red!10!blue!20!gray!30!white, draw=black, thick] (0,0) circle (3cm);
  \draw[thick] (-1.7,1) to[out=30,in=150] (-1.2,1.1) to[out=-30,in=210] (-0.8,1);
  \node at (-1.2,1.4) {${\scriptstyle x_1}$};
  \draw[thick] (0.5,1.7) to[out=60,in=120] (0.9,1.7) to[out=-60,in=240] (1.3,1.6);
  \node at (0.9,2) {${\scriptstyle x_2}$};
  \draw[thick,rotate around={20:(0,0)}] (0,0.2) to[out=60,in=120] (0.4,0.2) to[out=-60,in=240] (0.8,0.1);
  \node at (0.4,0.6) {${\scriptstyle x_3}$};
  \draw[thick] (-1.5,-1.2) to[out=10,in=170] (-1,-1.3) to[out=-10,in=190] (-0.5,-1.2);
  \node at (-1,-0.9) {${\scriptstyle x_n}$};
  \node[scale=.5] at (-.1,-.5) {$\cdots$};
  \node[scale=1.05] at (1.8,-1) {$D$};
\end{tikzpicture}
\quad\begin{matrix}
\text{with intervals}
\\
I_1,\ldots, I_n \subset D\\
\text{and}\\\text{algebra elements}
\\
x_i \in \cA(I_i)
\end{matrix}
\,\right]
\quad\, \mapsto \,\quad |x_1 \ldots x_n \rangle \in H_0(D).
\end{equation}
We note that while the article \cite{HenriquesTenerWorms} takes finite-dimensionality of $L_0$-eigenspaces as a blanket assumption for conformal nets, the constructions in Section 5 of that article do not require that assumption, which is not used until later sections of the article.

Let $D$ be a disc.
An embedded (oriented) interval $I \subset D$ is called \emph{extendable} if there exists a larger interval $I_+ \subset D$ containing $I$ in its interior, and if the orientations of $I$ and $\partial D$ agree whenever the two intersect.
Given a family $I_1, \ldots, I_n \subset D$ of disjoint, extendable intervals, and algebra elements $x_j \in \cA(I_j)$, we constructed in \cite[Eqn. (42)]{HenriquesTenerWorms} a vector
\[
| x_1 \cdots x_n \rangle_D \in H_0(D).
\]
When necessary for brevity, we will sometimes denote a collection of disjoint extendable intervals by $\underline{I} := (I_j)_{j=1}^n$, and corresponding algebra elements by 
\[
\underline{x} =(x_1, \ldots, x_n) \in \cA(\underline{I}) := \cA(I_1) \times \cdots \times \cA(I_n).
\]
Similarly, if $\underline{I}$ is a collection of disjoint extendable intervals in a disc $D$, and $\underline{x} \in \cA(\underline{I})$ is a choice of labels by algebra elements, then we sometimes write
\[
|\underline{x}\rangle_D := | x_1 \cdots x_n \rangle_D.
\]
More generally, if $\underline{I}_1, \ldots, \underline{I}_m$ are collections of extendable intervals in $D$ which are collectively pairwise disjoint, and  $\underline{x}_1, \ldots, \underline{x}_m$ are a family of labels by algebra elements $\underline{x}_j \in \cA(\underline{I}_j)$, then we give
\[
| \underline{x}_1 \cdots \underline{x}_m \rangle_D
\]
the obvious meaning of performing all of the insertions from all of the collections $\underline{x}_j$.

If $D'$ is another disc, $\varphi:D \to D'$ is a biholomorphic map, and $U_\varphi:H_0(D) \to H_0(D')$ is the associated unitary operator, then we have
\begin{equation}\label{eqn: biholomorphic map acting on disc insertions}
    U_\varphi| x_1 \cdots x_n \rangle_{D} = | \varphi_*(x_1) \cdots \varphi_*(x_n) \rangle_{D'}.
\end{equation}

Similarly, given an annulus $A = D_{out} \setminus \mathring D_{in}$, a family $(I_j)_{j=1}^n$ of disjoint extendable intervals in $A$, and algebra elements $x_j \in \cA(I_j)$ we constructed in \cite[Eqn. (44)]{HenriquesTenerWorms} a bounded operator
\begin{equation}\label{eqn: annulus with worm insertions}
A[x_1 \cdots x_n]:H_0(D_{in}) \to H_0(D_{out}).
\end{equation}
As with discs, we will sometimes denote a collection of intervals by $\underline I := (I_1, \ldots, I_n)$, and given algebra elements $\underline x \in \cA(\underline I)$ we will write the operator \eqref{eqn: annulus with worm insertions} as $A[\underline x]$.
In the absence of worm insertions, the operator \eqref{eqn: annulus with worm insertions} reduces to \eqref{eqn: def YA}.
If $J_1, \ldots, J_m$ is a family of extendable intervals in $D_{in}$, labelled by algebra elements $y_j \in \cA(J_j)$, then we have by \cite[Eqn. (48)]{HenriquesTenerWorms}
\begin{equation}\label{eqn: annuli with worms acting on disc with worms}
A[x_1 \cdots x_n] | y_1 \cdots y_m \rangle_{D_{in}} = | x_1 \cdots x_n \, y_1 \cdots y_m \rangle_{D_{out}}.
\end{equation}
Without the worm insertions $x_j$, this reduces to:
\begin{equation}\label{eqn: annuli without worms acting on discs with worms}
Y_A | y_1 \cdots y_m \rangle_{D_{in}} = | y_1 \cdots y_m \rangle_{D_{out}}.
\end{equation}
By \cite[Lem. 5.5]{HenriquesTenerWorms}, whenever $I$ is an interval contained in $\partial D_{in} \cap \partial D_{out}$, the operator $A[x_1 \cdots x_n]$ is equivariant for the actions of $\cA(I)$ on $H_0(D_{in})$ and $H_0(D_{out})$.

Similarly, if $A \in \Ann$ is an annulus, $\underline A \in \Ann_c$ is a lift of $A$, $I_1, \ldots, I_n \subset A$ are disjoint extendable intervals, and $x_j \in \cA(I_j)$, then we constructed in \cite[Eqn. (45)]{HenriquesTenerWorms} an operator
\begin{equation}\label{eqn: parametrised annulus with worm insertions}
\underline{A}[x_1 \cdots x_n]:H_0 \to H_0.
\end{equation}
If $A$ is written $A = D_{out} \setminus \mathring D_{in}$, with boundary parametrisations $\varphi_{in/out}: \partial \bbD \to \partial D_{in/out}$, then by \cite[Lem. 5.3]{HenriquesTenerWorms} we have
\begin{equation}\label{eqn: parametrised vs unparametrised annuli}
U_{\varphi_{out}} \underline A[x_1 \cdots x_n] U_{\varphi_{in}}^* = c \cdot A[x_1 \cdots x_n]
\end{equation}
where $U_{\varphi_{in/out}}:H_0 \to H_0(D_{in/out})$ are any unitaries implementing $\varphi_{in/out}$, and $c \in \bbC^\times$ is a scalar that does not depend on the choice of worm insertions $x_j$.

Suppose now that $A' \in \Ann$ is another annulus written as $A' = D_{out}' \setminus \mathring D_{in}'$ with boundary parametrisations $\varphi_{in/out}^\prime: \partial \bbD \to \partial D_{in/out}'$, and
suppose that $f:A \to A'$ is an isomorphism, meaning that $f$ is a homeomorphism, $f|_{\mathring A}$ is holomorphic, and $\varphi'_{in/out}=f\circ \varphi_{in/out}$.
As described in \cite[\S2]{HenriquesTener24ax}, the annuli $A$ and $A'$ lie in the same equivalence class in $\Ann$, and thus any pair of lifts $\underline A, \underline A' \in \Ann_c$ are scalar multiples of one another.
It follows that for any collection of worm insertions $x_j$ in $A$, the operators $\underline A[x_1 \cdots x_n]$ and $\underline A' [f_*(x_1) \cdots f_*(x_n)]$ on $H_0$ differ by the same scalar multiple.
In light of \eqref{eqn: parametrised vs unparametrised annuli}, we therefore have
\begin{equation}\label{eqn: comparison of unparametrised annuli under iso}
U_{f_{out}}A[x_1 \cdots x_n]U_{f_{in}}^* = c \cdot  A'[f_*(x_1) \cdots f_*(x_n)] ,
\end{equation}
where $f_{in/out} = f|_{\partial D_{in/out}}$ regarded as a diffeomorphism $\partial D_{in/out} \to \partial D_{in/out}'$, and  $c \in \bbC^\times$ is a scalar that does not depend on the choice of worm insertions.

\section{Genus zero functorial CFT from conformal nets}

Throughout this section, we fix a conformal net $\cA$ with vacuum Hilbert space $H_0$, and consider its coordinate-free extension as described in Section~\ref{sec: conformal nets}.
The goal of the section is to construct the genus zero functorial CFT corresponding to the vacuum sector of the conformal net $\cA$.
Recall (Definition~\ref{defn: operad of conformal discs}) that $\mathfrak D(n)$ denotes the collection of multidisc embeddings $\mathfrak d:D_1 \sqcup \cdots \sqcup D_n \to D$, and $\mathfrak D = \bigsqcup_{n=0}^\infty \mathfrak D(n)$ denotes the operad of conformal discs.
We will construct an algebra for this operad by showing that for each multidisc embedding $\mathfrak d:D_1 \sqcup \cdots \sqcup D_n \to D$, there is a bounded operator 
\[
Y_{\mathfrak d}:H_0(D_1) \otimes \cdots \otimes H_0(D_n) \to H_0(D)
\]
with the following property: for each $j=1, \ldots, n$ let $\underline I_j$ be a collection of disjoint extendable intervals in $D_j$, and let $\underline x_j \in \cA(\underline I_j)$,
then we have 
\begin{equation}\label{eqn: vacuum cft characterising property section intro}
Y_{\mathfrak d}(|\underline x_1\rangle_{D_1} \otimes \cdots \otimes |\underline x_n\rangle_{D_n}) = | \mathfrak d_*(\underline x_1) \cdots \mathfrak d_*(\underline x_n) \rangle_D.
\end{equation}
In the case $n=0$, this should be interpreted to say that the map $Y_{D}:\bbC \to H_0(D)$ corresponding to the empty embedding $\emptyset \to D$ assigns to the number $1 \in \bbC$ the vector $\Omega_D = | \, \rangle_D$.
Such maps, if they exist, evidently satisfy the necessary composition law to be an algebra for the operad of conformal discs, so the primary content of the section is to demonstrate the existence of \emph{bounded} operators $Y_{\mathfrak d}$ satisfying \eqref{eqn: vacuum cft characterising property section intro}.

We begin with a certain technical construction of bounded operators (Lemmas~\ref{lem: XSigma doesnt depend on presentation} and \ref{lem: Omega in X Sigma}) corresponding to inclusions of discs.
We will show later (in the proof of the main result Theorem~\ref{thm: Yd exists}) that these lemmas furnish the maps $Y_{\mathfrak d}$ when $\mathfrak d \in \mathfrak D(2)$.

Recall that, given an inclusion $A\subset B$ of topological spaces, the \emph{relative interior} of $A$ inside $B$, denoted $\mathrm{Rel\text{-}Int}(A)$, is the interior of $A$ viewed as a subset of $B$. 
\begin{defn}\label{def: star}
A \emph{four-pointed star}, or \emph{diamond}, (also called \emph{`ninja-star'} in \cite{Henriques14}) is a closed  contractible subspace of $\bbC$ of the form
\begin{equation}\label{eq: def diamond}
\Sigma \;=\; D_3 \setminus \relint(D_1 \sqcup D_2)
\end{equation}
where $D_1,D_2\subset \bbC$ are disjoint discs contained in some bigger disc $D_3$ and sharing part of its boundary, and the relative interior is taken with respect to $D_3$.
(The condition that \eqref{eq: def diamond} be contractible implies that $\partial D_1\cap \partial D_3$ and $\partial D_2\cap \partial D_3$ are intervals.)

The subsets $\partial D_1 \cap \Sigma$ and $\partial D_2 \cap \Sigma$ are called the \emph{incoming intervals} of $\Sigma$, and the connected components of $\partial D_3 \cap \Sigma$ are called the \emph{outgoing intervals} of $\Sigma$.
The decomposition of $\partial \Sigma$ into incoming and outgoing intervals is part of the data of the four-pointed star.
Given a four-pointed star $\Sigma$, a choice of discs $D_1$, $D_2$, and $D_3$ satisfying \eqref{eq: def diamond} is called a \emph{presentation} of $\Sigma$.
\end{defn}

The following is an example of a four-pointed star:
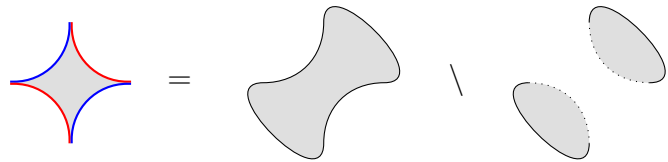
\[
\def\h{.25}\def\l{3.2}
\tikz[scale=\h, baseline=-1.88]{
\fill[gray!25] ($(90:.014/\h) +(0:3)$) to [in=-90, out=180] ($(90+90:-.014/\h) +(90:3)$)
-- ($(90+90:.014/\h) +(90:3)$) to [in=-90+90, out=180+90] ($(90+180:-.014/\h) +(180:3)$)
-- ($(90+180:.014/\h) +(180:3)$) to [in=-90+180, out=180+180] ($(90+270:-.014/\h) +(270:3)$)
-- ($(90+270:.014/\h) +(270:3)$) to [in=-90+270, out=180+270] ($(90+360:-.014/\h) +(360:3)$) -- cycle;
\draw[thick, red] (90:.014/\h) +(0:\l) -- +(0:3) to [in=-90, out=180]
($(90+90:-.014/\h) +(90:3)$) -- ($(90+90:-.014/\h) +(90:\l)$);
\draw[rotate=90, thick, blue] (90:.014/\h) +(0:\l) -- +(0:3) to [in=-90, out=180]
($(90+90:-.014/\h) +(90:3)$) -- ($(90+90:-.014/\h) +(90:\l)$);
\draw[rotate=180, thick, red] (90:.014/\h) +(0:\l) -- +(0:3) to [in=-90, out=180]
($(90+90:-.014/\h) +(90:3)$) -- ($(90+90:-.014/\h) +(90:\l)$);
\draw[rotate=-90, thick, blue] (90:.014/\h) +(0:\l) -- +(0:3) to [in=-90, out=180]
($(90+90:-.014/\h) +(90:3)$) -- ($(90+90:-.014/\h) +(90:\l)$);
}
\,\,\,\,\,=\,
\tikz[scale=\h, baseline=-1.88]{
\filldraw[fill=gray!25]
(0,3.2) ..controls +(0,3) and +(3,0).. (3.2,0)
..controls +(-1.9,0) and +(0,1.9).. (0,-3.2)
..controls +(0,-3) and +(-3,0).. (-3.2,0)
..controls +(1.9,0) and +(0,-1.9).. (0,3.2);
}
\setminus
\tikz[scale=\h, baseline=-1.88]{
\fill[gray!25]
(0,3.2) ..controls +(0,3) and +(3,0).. (3.2,0)
..controls +(-1.9,0) and +(0,-1.9).. (0,3.2);
\fill[gray!25]
(0,-3.2) ..controls +(0,-3) and +(-3,0).. (-3.2,0)
..controls +(1.9,0) and +(0,1.9).. (0,-3.2);
\draw
(0,3.2) ..controls +(0,3) and +(3,0).. (3.2,0);
\draw[dotted]
(3.2,0) ..controls +(-1.9,0) and +(0,-1.9).. (0,3.2);
\draw
(0,-3.2) ..controls +(0,-3) and +(-3,0).. (-3.2,0);
\draw[dotted]
(-3.2,0) ..controls +(1.9,0) and +(0,1.9).. (0,-3.2);
}
\]
The incoming intervals are drawn in red, and outgoing ones are drawn in blue.

Given a four-pointed star $\Sigma$
written as \eqref{eq: def diamond}
with $D_1$, $D_2$, $D_3$ all contained in $\bbD$,
let
\[
A:=\bbD \setminus \mathring D_3,
\qquad
J_1:= \partial D_1 \cap \partial D_3,\qquad \text{and}\qquad
J_2:= \partial D_2 \cap \partial D_3.
\]
By the split property of the conformal net, both $H_0(D_1) \otimes H_0(D_2)$ and $H_0(D_3)$ carry representations of the spatial tensor product $\cA(J_1) \otimes \cA(J_2)$.
\begin{defn}\label{def: def of X(Sigma)}
We define the subspace
\begin{equation}\label{eq: X(Sigma) c H_0}
X(\Sigma)\subset H_0
\end{equation}
to be the image under $Y_A:H_0(D_3)\to H_0$ of
\[
\Hom_{\cA(J_1)\otimes \cA(J_2)}\Big(H_0(D_1)\otimes H_0(D_2), H_0(D_3)\Big) (\Omega_1\otimes \Omega_2)
\subset H_0(D_3)
\]
(where, $\Omega_1\in H_0(D_1)$ and $\Omega_2\in H_0(D_2)$ are the respective vacuum vectors).
We equip $X(\Sigma)$ with the topology of pointwise convergence of operators, i.e., the subspace topology of the strong operator topology on 
$\Hom(H_0(D_1)\otimes H_0(D_2), H_0(D_3))$.
\end{defn}

\begin{rem}
Given a four-pointed star $\Sigma$, let us denote by $\Sigma^{\mathrm{rev}}$ the four-pointed star with the same underlying space and the same orientation, but where the incoming and outgoing boundaries are switched.
If $\Sigma\subset\mathring \bbD$,
we thus get two subspaces $X(\Sigma)$ and $X(\Sigma^{\mathrm{rev}})$ of $H_0$. We expect that when $\cA$ is not  rational (i.e. when $\cA$ has infinite $\mu$-index in the sense of \cite{KaLoMu01}), we have $X(\Sigma) \ne X(\Sigma^{\mathrm{rev}})$.
\end{rem}

\begin{lem}\label{lem: XSigma doesnt depend on presentation}
    The subspace $X(\Sigma) \subset H_0(\bbD)$ only depends on $\Sigma$ and not on its presentation, i.e., not on the choices of discs $D_1$ and $D_2$ (the third disc, $D_3$, can be recovered as $D_3=\Sigma \cup D_1\cup D_2$).
    Similarly, the topology on $X(\Sigma)$ described above only depends on $\Sigma$.
\end{lem}
\begin{proof}
To emphasize the potential dependence on $D_1$ and $D_2$, let us write $X(\Sigma, D_1, D_2)$ for the subset of $H_0(\bbD)$ defined in \eqref{eq: X(Sigma) c H_0}.
It suffices to show that $X(\Sigma, D_1, D_2)$ is unchanged when the discs $D_i$ are replaced by smaller discs $D_i' \subset D_i$ as for any two choices $(D_1, D_2)$ and $(D'_1, D'_2)$ of discs there exists a third choice $(D''_1, D''_2)$ such that $D_i''\subset D_i$ and $D_i''\subset D_i'$.
In fact, it suffices to replace only a single disc at a time ($D_1$, say), leaving the other unchanged.

Pick a presentation
$\Sigma = D_3 \setminus  \relint(D_1 \sqcup D_2)$
of $\Sigma$ as in \eqref{eq: def diamond}, with $D_3\subset \bbD$, and
let $D_1'\subset D_1$ be a smaller disc so that $D_1'$, $D_2$ together with $D_3'=\Sigma \cup D_1'\cup D_2$ form another presentation of $\Sigma$:
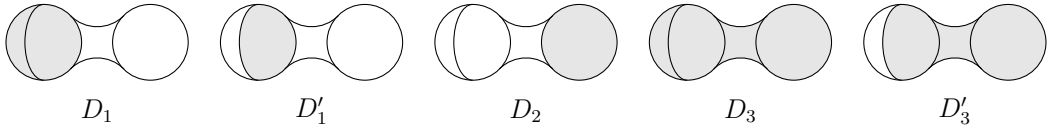
\vspace{2mm}
\[
\tikz[scale=.5, baseline=0]{
\draw[fill=gray!20] (-1.41,0) circle (1);
\draw (1.41,0) circle (1);
\draw (-.705,.705) arc (-135:-45:1);
\draw (-.705,-.705) arc (135:45:1);
\draw (-1.41,1) arc (90:270:.5 and 1);
\node[scale=.85] at (0,-1.8) {$D_1$};
}
\quad
\tikz[scale=.5, baseline=0]{
\fill[fill=gray!20] (-1.41,1) arc (90:270:.5 and 1) arc (-90:90:1); 
\draw (-1.41,0) circle (1);
\draw (1.41,0) circle (1);
\draw (-.705,.705) arc (-135:-45:1);
\draw (-.705,-.705) arc (135:45:1);
\draw (-1.41,1) arc (90:270:.5 and 1);
\node[scale=.85] at (0,-1.8) {$D'_1$};
}
\quad
\tikz[scale=.5, baseline=0]{
\draw (-1.41,0) circle (1);
\draw[fill=gray!20] (1.41,0) circle (1);
\draw (-.705,.705) arc (-135:-45:1);
\draw (-.705,-.705) arc (135:45:1);
\draw (-1.41,1) arc (90:270:.5 and 1);
\node[scale=.85] at (0,-1.8) {$D_2$};
}
\quad
\tikz[scale=.5, baseline=0]{
\fill[fill=gray!20] (-.705,.705) arc (-135:-45:1) -- (.705,-.705) arc (-135+180:-45+180:1);
\draw[fill=gray!20] (-1.41,0) circle (1);
\draw[fill=gray!20] (1.41,0) circle (1);
\draw (-.705,.705) arc (-135:-45:1);
\draw (-.705,-.705) arc (135:45:1);
\draw (-1.41,1) arc (90:270:.5 and 1);
\node[scale=.85] at (0,-1.8) {$D_3$};
}
\quad
\tikz[scale=.5, baseline=0]{
\fill[fill=gray!20] (-.705,.705) arc (-135:-45:1) -- (.705,-.705) arc (-135+180:-45+180:1);
\fill[fill=gray!20] (-1.41,1) arc (90:270:.5 and 1) arc (-90:90:1); 
\draw (-1.41,0) circle (1);
\draw[fill=gray!20] (1.41,0) circle (1);
\draw (-.705,.705) arc (-135:-45:1);
\draw (-.705,-.705) arc (135:45:1);
\draw (-1.41,1) arc (90:270:.5 and 1);
\node[scale=.85] at (0,-1.8) {$D'_3$};
}\vspace{-2mm}
\]
Let\vspace{-1mm}
\begin{gather*}
J_1 = \partial D_1 \cap \partial D_3,\quad
J_1' = \partial D_1' \cap \partial D_3',\quad
J_2 = \partial D_2 \cap \partial D_3,
\\
A=\bbD\setminus \mathring D_3,\qquad A'=\bbD\setminus \mathring D'_3,
\\
T\,:=\Hom_{\cA(J_1)\otimes \cA(J_2)}\big(H_0(D_1)\otimes H_0(D_2), H_0(D_3)\big),\\
T':=\Hom_{\cA(J'_1)\otimes \cA(J_2)}\big(H_0(D'_1)\otimes H_0(D_2), H_0(D_3')\big).
\end{gather*}
Our task is to show that
$X(\Sigma, D_1, D_2) = Y_A T(\Omega_1\otimes \Omega_2)$ and 
$X(\Sigma, D'_1, D_2) = Y_{A'} T'(\Omega'_1\otimes \Omega_2)
$ are homeomorphic, and equal as subsets of $H_0$ (i.e.: the identity map is a homeomorphism from $X(\Sigma, D_1, D_2)$ to $X(\Sigma, D'_1, D_2)$). 

Let $B_1=D_1 \setminus \mathring D_1'$ and $B_3=D_3 \setminus \mathring D_3'$, and note that the interiors of these annuli coincide.
We have corresponding operators
\[
Y_{B_1}:H_0(D_1') \to H_0(D_1), \qquad Y_{B_3}:H_0(D_3') \to H_0(D_3)
\]
as in~\eqref{eqY}.

We claim that there exists a linear isomorphism $T \cong T'$
which identifies $x \in T$ and $x' \in T'$ whenever
\begin{equation}\label{eqn: relationship between x and x prime}
Y_{B_3}x' = x(Y_{B_1} \otimes 1).
\end{equation}
That is, we claim that for every $x \in T$, there exists a unique $x' \in T'$ such that \eqref{eqn: relationship between x and x prime} holds, and conversely for every $x'$ there exists a unique $x$ satisfying that same relation.
The uniqueness part of these claims follow from the fact that $Y_{B_1}$ and $Y_{B_3}$ are injective with dense image (\cite[Lem 4.15]{HenriquesTenerWorms}). It remains to show existence.

For $j=1,3$, pick diffeomorphisms $\gamma_j:\partial D_j\to S^1$ and $\gamma'_j:\partial D'_j\to S^1$ satisfying
\begin{align*}
\gamma_1&=\gamma'_1\quad
\text{on}\quad \partial D_1 \cap \partial D_1'
\quad
&\gamma_3&=\gamma'_3\quad
\text{on}\quad \partial D_3 \cap \partial D_3'
\\
\gamma_1&=\gamma_3\quad
\text{on}\quad \partial D_1 \cap \partial D_3
\quad
&\gamma'_1&=\gamma'_3\quad
\text{on}\quad \partial D'_1 \cap \partial D_3'.
\end{align*}
Let $v_j: H_0(D_j)\to H_0(\bbD)$ and $v'_j: H_0(D'_j)\to H_0(\bbD)$ be corresponding implementing unitaries (which are unique up to phase), as described in Section~\ref{sec: conformal nets}, and let 
\[
u_j:=v'^*_jv_j:H_0(D_j) \to H_0(D_j').
\]
The elements $v_1Y_{B_1}v_1'^*$ and $v_3Y_{B_3}v_3'^*\in\Ann_c$ map to the same place under the projection $\Ann_c\twoheadrightarrow \Ann$, so there exists a scalar $c\in \bbC^\times$ such that
\[
v_1Y_{B_1}v_1'^*
=
c\cdot v_3Y_{B_3}v_3'^*
=: \underline A.
\]
In light of the fact that $\gamma_1|_{J_1} = \gamma_3|_{J_1}$, this operator $\underline A \in B(H_0(\bbD))$ is localised in $\cA(K)$ for $K:=\gamma_1(J_1) =\gamma_3(J_1)$.
Thus there is a single element of $\cA(J_1)$ whose action on $H_0(D_1)$ is given by
\[
{\gamma_1}^{-1}_*(v_1Y_{B_1}v_1'^*) = Y_{B_1}u_1
\]
and whose action on $H_0(D_3)$ is given by
\[
c\cdot {\gamma_3}^{-1}_*( v_3Y_{B_3}v_3'^*) = c\cdot Y_{B_3}u_3.
\]

If $x \in T$, then set $x' := c\cdot u_3x(u_1^* \otimes 1)\in T'$.
Since $x$ intertwines the actions of $\cA(J_1)$, our above calculation shows that
\[
c\cdot Y_{B_3}u_3x = x(Y_{B_1}u_1 \otimes 1),
\]
from which the identity \eqref{eqn: relationship between x and x prime} immediately follows.
Conversely, starting with $x' \in T'$ we may set $x := c^{-1}\cdot u_3^* x'(u_1 \otimes 1)$, and again $x$ and $x'$ satisfy \eqref{eqn: relationship between x and x prime}.
We have thus established the existence of an isomorphism $T \cong T'$ characterized by this identity. Moreover, this isomorphism is visibly compatible with the strong operator topologies on those two spaces.

Finally, if $x$ and $x'$ satisfy \eqref{eqn: relationship between x and x prime}, then
\[
Y_Ax(\Omega_1 \otimes \Omega_2) = Y_A x (Y_{B_1}\Omega_1' \otimes \Omega_2)
= Y_A Y_{B_3} x'(\Omega_1' \otimes \Omega_2) = Y_{A'} x'(\Omega_1' \otimes \Omega_2).
\]
It follows that
$X(\Sigma,D_1,D_2) = Y_AT(\Omega_1 \otimes \Omega_2)$ and 
$Y_{A'}T'(\Omega_1'\otimes \Omega_2) = X(\Sigma, D_1', D_2)$
are equal as subsets of $H_0$,
and that $x\in T$ and $x'\in T'$ map to the same element of $H_0$ when $x$ and $x'$ are related by \eqref{eqn: relationship between x and x prime}.
\end{proof}

\begin{lem}\label{lem: Omega in X Sigma}
    Let $\Sigma \subset \mathring \bbD$ be a four-pointed star.
    \begin{enumerate}[i)]
        \item $X(\Sigma)$ is a dense subspace of $H_0$ for the Hilbert space topology
        \item If $\Sigma' \subset \mathring \Sigma$ is another four-pointed star, then $X(\Sigma') \subset X(\Sigma)$.
        \item $\Omega \in X(\Sigma)$.
    \end{enumerate}
\end{lem}
\begin{proof}
(i). Pick a presentation
$\Sigma = D_3 \setminus  \relint(D_1 \sqcup D_2)$
of $\Sigma$ with $D_3\subset \bbD$, and
let $A := \bbD \setminus \mathring D_3$ so that 
\[
X(\Sigma) = \Big\{Y_Ax(\Omega_1 \otimes \Omega_2) \, : \, x \in \Hom_{\cA(J_1)\otimes \cA(J_2)}\big(H_0(D_1)\otimes H_0(D_2), H_0(D_3)\big) \Big\} \subset H_0,
\]
where $J_1 = \partial D_1 \cap \partial D_3$ and $J_2 = \partial D_2 \cap \partial D_3$.
We use the following well-known fact from the general theory of von Neumann algebras: If $M\subset B(H)$ is a von Neumann algebra and $\Omega \in H$ is a separating vector, then for any representation $K$ of $M$ the space $\Hom_M(H, K)\Omega$ is dense in $K$.
Statement (i) now follows immediately from this as $\Omega_1 \otimes \Omega_2$ is separating for $\cA(J_1)\otimes \cA(J_2)$ (by the Reeh-Schlieder theorem) and $Y_A$ has dense image (\cite[Lem 4.15]{HenriquesTenerWorms}).

(ii).
Let $D_j$, $A$, and $J_j$ be as above, and let $\Sigma' \subset \mathring \Sigma$ be another four-pointed star.
We can then choose a presentation $\Sigma' = D_3' \setminus \relint(D_1'\sqcup D_2')$ of $\Sigma'$ such that $D_1 \subseteq D_1'$, $D_2 \subseteq D_2'$, and $D_3' \subseteq D_3$:
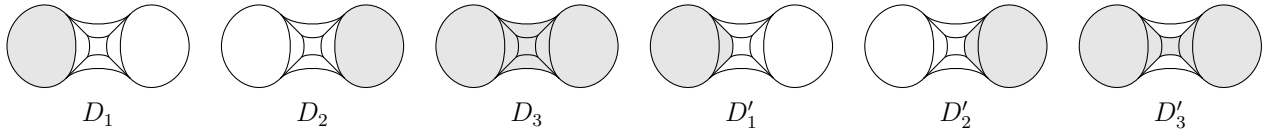
\[
\hspace{-1mm}
\tikz[scale=.5, baseline=0]{
\fill[fill=gray!20] (-.85,.85) ..controls +(-.6,.6) and +(0,.85).. (-2.41,0) ..controls +(0,-.85) and +(-.6,-.6).. (-.85,-.85) ..controls +(.35,.35) and +(.35,-.35).. (-.85,.85);
\draw (-.3,.3) -- (-.85,.85) ..controls +(-.6,.6) and +(0,.85).. (-2.41,0) ..controls +(0,-.85) and +(-.6,-.6).. (-.85,-.85) -- (-.3,-.3);
\draw (.3,.3) -- (.85,.85) ..controls +(.6,.6) and +(0,.85).. (2.41,0) ..controls +(0,-.85) and +(.6,-.6).. (.85,-.85) -- (.3,-.3);
\draw (-.85,.85) ..controls +(.35,-.35) and +(-.35,-.35).. (.85,.85);
\draw (-.85,-.85) ..controls +(.35,.35) and +(-.35,.35).. (.85,-.85);
\draw (-.85,.85) ..controls +(.35,-.35) and +(.35,.35).. (-.85,-.85);
\draw (.85,.85) ..controls +(-.35,-.35) and +(-.35,.35).. (.85,-.85);
\draw (-.3,.3) ..controls +(.1,-.1) and +(-.1,-.1).. (.3,.3);
\draw (-.3,-.3) ..controls +(.1,.1) and +(-.1,.1).. (.3,-.3);
\draw (-.3,.3) ..controls +(.1,-.1) and +(.1,.1).. (-.3,-.3);
\draw (.3,.3) ..controls +(-.1,-.1) and +(-.1,.1).. (.3,-.3);
\node[scale=.85] at (0,-1.8) {$D_1$};
}
\quad
\tikz[scale=.5, baseline=0]{
\fill[fill=gray!20] (.85,.85) ..controls +(.6,.6) and +(0,.85).. (2.41,0) ..controls +(0,-.85) and +(.6,-.6).. (.85,-.85) ..controls +(-.35,.35) and +(-.35,-.35).. (.85,.85);
\draw (-.3,.3) -- (-.85,.85) ..controls +(-.6,.6) and +(0,.85).. (-2.41,0) ..controls +(0,-.85) and +(-.6,-.6).. (-.85,-.85) -- (-.3,-.3);
\draw (.3,.3) -- (.85,.85) ..controls +(.6,.6) and +(0,.85).. (2.41,0) ..controls +(0,-.85) and +(.6,-.6).. (.85,-.85) -- (.3,-.3);
\draw (-.85,.85) ..controls +(.35,-.35) and +(-.35,-.35).. (.85,.85);
\draw (-.85,-.85) ..controls +(.35,.35) and +(-.35,.35).. (.85,-.85);
\draw (-.85,.85) ..controls +(.35,-.35) and +(.35,.35).. (-.85,-.85);
\draw (.85,.85) ..controls +(-.35,-.35) and +(-.35,.35).. (.85,-.85);
\draw (-.3,.3) ..controls +(.1,-.1) and +(-.1,-.1).. (.3,.3);
\draw (-.3,-.3) ..controls +(.1,.1) and +(-.1,.1).. (.3,-.3);
\draw (-.3,.3) ..controls +(.1,-.1) and +(.1,.1).. (-.3,-.3);
\draw (.3,.3) ..controls +(-.1,-.1) and +(-.1,.1).. (.3,-.3);
\node[scale=.85] at (0,-1.8) {$D_2$};
}
\quad
\tikz[scale=.5, baseline=0]{
\fill[fill=gray!20] (-.85,.85) ..controls +(-.6,.6) and +(0,.85).. (-2.41,0) ..controls +(0,-.85) and +(-.6,-.6).. (-.85,-.85) 
..controls +(.35,.35) and +(-.35,.35)..
(.85,-.85) ..controls +(.6,-.6) and +(0,-.85).. (2.41,0) ..controls +(0,.85) and +(.6,.6).. (.85,.85) 
..controls +(-.35,-.35) and +(.35,-.35)..
(-.85,.85)
;
\draw (-.3,.3) -- (-.85,.85) ..controls +(-.6,.6) and +(0,.85).. (-2.41,0) ..controls +(0,-.85) and +(-.6,-.6).. (-.85,-.85) -- (-.3,-.3);
\draw (.3,.3) -- (.85,.85) ..controls +(.6,.6) and +(0,.85).. (2.41,0) ..controls +(0,-.85) and +(.6,-.6).. (.85,-.85) -- (.3,-.3);
\draw (-.85,.85) ..controls +(.35,-.35) and +(-.35,-.35).. (.85,.85);
\draw (-.85,-.85) ..controls +(.35,.35) and +(-.35,.35).. (.85,-.85);
\draw (-.85,.85) ..controls +(.35,-.35) and +(.35,.35).. (-.85,-.85);
\draw (.85,.85) ..controls +(-.35,-.35) and +(-.35,.35).. (.85,-.85);
\draw (-.3,.3) ..controls +(.1,-.1) and +(-.1,-.1).. (.3,.3);
\draw (-.3,-.3) ..controls +(.1,.1) and +(-.1,.1).. (.3,-.3);
\draw (-.3,.3) ..controls +(.1,-.1) and +(.1,.1).. (-.3,-.3);
\draw (.3,.3) ..controls +(-.1,-.1) and +(-.1,.1).. (.3,-.3);
\node[scale=.85] at (0,-1.8) {$D_3$};
}
\quad
\tikz[scale=.5, baseline=0]{
\fill[fill=gray!20] (-.3,.3) -- (-.85,.85) ..controls +(-.6,.6) and +(0,.85).. (-2.41,0) ..controls +(0,-.85) and +(-.6,-.6).. (-.85,-.85) -- (-.3,-.3) ..controls +(.1,.1) and +(.1,-.1).. (-.3,.3);
\draw (-.3,.3) -- (-.85,.85) ..controls +(-.6,.6) and +(0,.85).. (-2.41,0) ..controls +(0,-.85) and +(-.6,-.6).. (-.85,-.85) -- (-.3,-.3);
\draw (.3,.3) -- (.85,.85) ..controls +(.6,.6) and +(0,.85).. (2.41,0) ..controls +(0,-.85) and +(.6,-.6).. (.85,-.85) -- (.3,-.3);
\draw (-.85,.85) ..controls +(.35,-.35) and +(-.35,-.35).. (.85,.85);
\draw (-.85,-.85) ..controls +(.35,.35) and +(-.35,.35).. (.85,-.85);
\draw (-.85,.85) ..controls +(.35,-.35) and +(.35,.35).. (-.85,-.85);
\draw (.85,.85) ..controls +(-.35,-.35) and +(-.35,.35).. (.85,-.85);
\draw (-.3,.3) ..controls +(.1,-.1) and +(-.1,-.1).. (.3,.3);
\draw (-.3,-.3) ..controls +(.1,.1) and +(-.1,.1).. (.3,-.3);
\draw (-.3,.3) ..controls +(.1,-.1) and +(.1,.1).. (-.3,-.3);
\draw (.3,.3) ..controls +(-.1,-.1) and +(-.1,.1).. (.3,-.3);
\node[scale=.85] at (0,-1.8) {$D'_1$};
}
\quad
\tikz[scale=.5, baseline=0]{
\fill[fill=gray!20] (.3,.3) -- (.85,.85) ..controls +(.6,.6) and +(0,.85).. (2.41,0) ..controls +(0,-.85) and +(.6,-.6).. (.85,-.85) -- (.3,-.3) ..controls +(-.1,.1) and +(-.1,-.1).. (.3,.3);
\draw (-.3,.3) -- (-.85,.85) ..controls +(-.6,.6) and +(0,.85).. (-2.41,0) ..controls +(0,-.85) and +(-.6,-.6).. (-.85,-.85) -- (-.3,-.3);
\draw (.3,.3) -- (.85,.85) ..controls +(.6,.6) and +(0,.85).. (2.41,0) ..controls +(0,-.85) and +(.6,-.6).. (.85,-.85) -- (.3,-.3);
\draw (-.85,.85) ..controls +(.35,-.35) and +(-.35,-.35).. (.85,.85);
\draw (-.85,-.85) ..controls +(.35,.35) and +(-.35,.35).. (.85,-.85);
\draw (-.85,.85) ..controls +(.35,-.35) and +(.35,.35).. (-.85,-.85);
\draw (.85,.85) ..controls +(-.35,-.35) and +(-.35,.35).. (.85,-.85);
\draw (-.3,.3) ..controls +(.1,-.1) and +(-.1,-.1).. (.3,.3);
\draw (-.3,-.3) ..controls +(.1,.1) and +(-.1,.1).. (.3,-.3);
\draw (-.3,.3) ..controls +(.1,-.1) and +(.1,.1).. (-.3,-.3);
\draw (.3,.3) ..controls +(-.1,-.1) and +(-.1,.1).. (.3,-.3);
\node[scale=.85] at (0,-1.8) {$D'_2$};
}
\quad
\tikz[scale=.5, baseline=0]{

\fill[fill=gray!20] (-.3,.3) -- (-.85,.85) ..controls +(-.6,.6) and +(0,.85).. (-2.41,0) ..controls +(0,-.85) and +(-.6,-.6).. (-.85,-.85) -- (-.3,-.3) 
..controls +(.1,.1) and +(-.1,.1)..
(.3,-.3) -- (.85,-.85) ..controls +(.6,-.6) and +(0,-.85).. (2.41,0) ..controls +(0,.85) and +(.6,.6).. (.85,.85) -- (.3,.3)
..controls +(-.1,-.1) and +(.1,-.1)..
(-.3,.3) -- (-.85,.85)
;

\draw (-.3,.3) -- (-.85,.85) ..controls +(-.6,.6) and +(0,.85).. (-2.41,0) ..controls +(0,-.85) and +(-.6,-.6).. (-.85,-.85) -- (-.3,-.3);
\draw (.3,.3) -- (.85,.85) ..controls +(.6,.6) and +(0,.85).. (2.41,0) ..controls +(0,-.85) and +(.6,-.6).. (.85,-.85) -- (.3,-.3);
\draw (-.85,.85) ..controls +(.35,-.35) and +(-.35,-.35).. (.85,.85);
\draw (-.85,-.85) ..controls +(.35,.35) and +(-.35,.35).. (.85,-.85);
\draw (-.85,.85) ..controls +(.35,-.35) and +(.35,.35).. (-.85,-.85);
\draw (.85,.85) ..controls +(-.35,-.35) and +(-.35,.35).. (.85,-.85);
\draw (-.3,.3) ..controls +(.1,-.1) and +(-.1,-.1).. (.3,.3);
\draw (-.3,-.3) ..controls +(.1,.1) and +(-.1,.1).. (.3,-.3);
\draw (-.3,.3) ..controls +(.1,-.1) and +(.1,.1).. (-.3,-.3);
\draw (.3,.3) ..controls +(-.1,-.1) and +(-.1,.1).. (.3,-.3);
\node[scale=.85] at (0,-1.8) {$D'_3$};
}
\]  
    Let $A' := \bbD \setminus \mathring D_3'$ and $B := D_3 \setminus \mathring D_3'$, so that $A' = A \cup B$.
    For $j \in \{1,2\}$, let $J_j' := \partial D_j' \cap \partial D_3'$, and let $C_j := D_j' \setminus \mathring D_j$. 
Since $J_j \subseteq J_j'$, we have $Y_{C_j} \in \Hom_{\cA(J_j)}(H_0(D_j), H_0(D'_j))$ by \cite[Lem. 4.19]{HenriquesTenerWorms}.
Similarly, $Y_B \in \Hom_{\cA(J_1) \otimes \cA(J_2)}(H_0(D_3'),H_0(D_3))$, hence for any $x \in\Hom_{\cA(J'_1)\otimes \cA(J'_2)}\big(H_0(D'_1)\otimes H_0(D'_2), H_0(D'_3)\big)$ we have
    \[
    Y_Bx(Y_{C_1} \otimes Y_{C_2}) \in \Hom_{\cA(J_1)\otimes \cA(J_2)}\big(H_0(D_1)\otimes H_0(D_2), H_0(D_3)\big).
    \]
    To finish the argument, we check that for any $x$ as above, the corresponding element of $X(\Sigma')$ lies in $X(\Sigma)$:
    \begin{equation}\label{eq: YYxYY}
    Y_{A'}x(\Omega_1' \otimes \Omega_2')
    =
    Y_{A}\big(Y_Bx(Y_{C_1} \otimes Y_{C_2})\big)(\Omega_1 \otimes \Omega_2)
    \in X(\Sigma).
    \end{equation}
    This completes the proof of (ii).
    
     
    
    (iii).
    If $\varphi:\bbD \to \bbD$ is a biholomorphic map
    and $U_\varphi:H_0 \to H_0$ is the associated unitary operator, then by construction we have $X(\varphi(\Sigma)) = U_\varphi X(\Sigma)$.
    We may thus assume without loss of generality that $0 \in \mathring \Sigma$ (as $U_\varphi\Omega = \Omega$ for any $\varphi$ as above).
    Pick $s > 0 $ with $s\bbD \subset \Sigma$, and choose a four-pointed star $\Sigma'$ contained in $s\mathring \bbD$.
    For any $t\in \bbR$, let $\Sigma'_t$ denote the result of applying a rotation of angle $t$ to $\Sigma'$:
    \[
    \Sigma'_t := e^{it}\Sigma' \subset \mathring \Sigma.
    \]
    The subspaces $X(\Sigma')$ and $X(\Sigma'_t)$ are related by $X(\Sigma'_t)=e^{itL_0}X(\Sigma')$, and so
    \begin{equation}\label{eq: ksjdgbkdn}
    \qquad\,\,\,\, e^{itL_0}X(\Sigma') \subset X(\Sigma),\qquad\forall t \in \bbR.
    \end{equation}


    Let $\Sigma = D_3 \setminus \relint(D_1 \sqcup D_2)$ be a presentation of $\Sigma$ with $D_3 \subseteq \bbD$, and let $A = \bbD \setminus \mathring D_3$.
    Recall that $V(0)=\bbC \Omega$ by the irreducibility axiom of the conformal net $\cA$.
    By (i), we may choose $\xi \in X(\Sigma')$ whose projection onto $V(0)$ is precisely $\Omega$.
    By \eqref{eq: ksjdgbkdn}, we have $e^{itL_0} \xi \in X(\Sigma)$ $\forall t \in \bbR$, and thus there exist $x_t \in \Hom_{\cA(J_1)\otimes \cA(J_2)}(H_0(D_1)\otimes H_0(D_2), H_0(D_3))$ such that
    \[
    e^{itL_0}\xi = Y_A x_t(\Omega_1 \otimes \Omega_2).
    \]
    Such operators $x_t$ are unique, as $\Omega_1 \otimes\Omega_2$ is cyclic for $\cA(J_1) \otimes \cA(J_2)$ and $Y_A$ is injective.
    If we let $x\in\Hom_{\cA(J'_1)\otimes \cA(J'_2)}(H_0(D'_1)\otimes H_0(D'_2), H_0(D'_3))$ be the operator corresponding to $\xi\in X(\Sigma')$, then $x_t$ is the image of $x$ under the maps
    \begin{align*}
    \Hom_{\cA(J'_1)\otimes \cA(J'_2)}&(H_0(D'_1)\otimes H_0(D'_2), H_0(D'_3))
    \\ &\cong X(\Sigma') \overset{e^{itL_0}}{\longrightarrow} X(\Sigma'_t)\hookrightarrow X(\Sigma)\cong \Hom_{\cA(J_1)\otimes \cA(J_2)}(H_0(D_1)\otimes H_0(D_2), H_0(D_3)).
    \end{align*}
    Since all $Y$ operators involved in \eqref{eq: YYxYY} are uniformly bounded in norm \cite[Thm. 6.4]{HenriquesTenerIntegratingax}, and multiplication of operators is jointly strong-operator-continuous on bounded sets, we see from the above description that $t\mapsto x_t$ is a continuous map $\bbR\to X(\Sigma)$ (i.e., continuous with respect to the strong operator topology on $\Hom(H_0(D_1)\otimes H_0(D_2), H_0(D_3))$). We may thus define $y := \tfrac{1}{2\pi} \int_0^{2\pi} x_t \, dt$. By construction, $y \in \Hom_{\cA(J_1)\otimes \cA(J_2)}(H_0(D_1)\otimes H_0(D_2), H_0(D_3))$.
    Moreover,
    \[
    Y_A y(\Omega_1 \otimes \Omega_2) =  \frac{1}{2\pi} \int_0^{2\pi} Y_A x_t(\Omega_1 \otimes \Omega_2) \, dt = \frac{1}{2\pi} \int_0^{2\pi} e^{itL_0}\xi \, dt = \Omega.
    \]
    Thus $\Omega \in X(\Sigma)$, as claimed.
\end{proof}

We are now ready to prove the main result of this section, establishing the existence of bounded maps satisfying \eqref{eqn: vacuum cft characterising property section intro} and establishing the existence of the functorial CFT corresponding to the vacuum sector of the conformal net $\cA$.

\begin{thm}\label{thm: Yd exists}
    Let $\mathfrak d:D_1 \sqcup \cdots \sqcup D_n \to D \in \mathfrak D$ be a multidisc embedding.
    Then there exists a unique bounded linear operator 
    \[
    Y_{\mathfrak d}: H_0(D_1) \otimes \cdots \otimes H_0(D_n) \to H_0(D)
    \]
    with the following property: for $j=1, \ldots, n$, whenever $\underline I_j$ are collections of disjoint extendable intervals in $D_j$, and $\underline x_j \in \cA(\underline I_j)$,
    we have 
    \begin{equation}\label{eqn: characterizing property of Yd}
Y_{\mathfrak d} \big( |\underline x_1\rangle_{D_1} \otimes \cdots \otimes |\underline x_n\rangle_{D_n} \big) = |\mathfrak d_*(\underline x_1) \cdots \mathfrak d_*(\underline x_n) \rangle_D.
    \end{equation}
    These maps form an algebra for the operad $\mathfrak D$ of conformal discs.
\end{thm}
    
    We note that, by requirement, if $J$ is an interval contained in $\partial (\mathfrak d(D_j)) \cap \partial D$ for some $j$, then $Y_{\mathfrak d}$ is equivariant for $x \in \cA(J)$, in that
    \[
     x\circ Y_{\mathfrak d} = Y_{\mathfrak d} \circ (\mathrm{id} \otimes \cdots \otimes \mathfrak d^{-1}_*(x) \otimes \cdots \otimes \mathrm{id}).
    \]
\begin{proof}

The uniqueness of operators satisfying \eqref{eqn: characterizing property of Yd} is an immediate consequence of the Reeh-Schlieder theorem.
If it is shown that bounded operators satisfying \eqref{eqn: characterizing property of Yd} exist for all $\mathfrak d \in \mathfrak D$, then they are evidently compatible with composition in $\mathfrak D$.
Moreover, since the operad $\mathfrak D$ is generated by $\mathfrak D(0)$ and $\mathfrak D(2)$, along with biholomorphic maps $\varphi:D_1 \to D \in \mathfrak D(1)$, it suffices to demonstrate the existence of operators $Y_{\mathfrak d}$ in these cases.
For a disc $D \in \mathfrak D(0)$, the necessary map is furnished by the vacuum vector $\Omega_D$.
For a biholomorphic map $\varphi:D_1 \to D$, the corresponding implementing unitary $U_\varphi:H_0(D_1) \to H_0(D)$ provided by the conformal net $\cA$ satisfies the necessary condition \eqref{eqn: characterizing property of Yd} (by \eqref{eqn: biholomorphic map acting on disc insertions}).
We are thus left to show existence of the bounded map $Y_{\mathfrak d}$ when $\mathfrak d:D_1 \sqcup D_2 \to D \in \mathfrak D(2)$.
Moreover, we may assume without loss of generality that $D_1,D_2 \subset D$ (i.e. that $\mathfrak d$ is the inclusion of two subdiscs).

    We next claim that it suffices to consider the case when $J_1 := \partial D_1 \cap \partial D$ and $J_2 := \partial D_2 \cap \partial D$ are intervals.
    If not, we may choose a subdisc $D' \subset D$ which contains $D_1$ and $D_2$, and such that $\partial D_1 \cap \partial D'$ and $\partial D_2 \cap \partial D'$ are intervals.
    Let $A = D \setminus \mathring D'$, and let $\mathfrak d'$ be the multidisc inclusion $D_1,D_2 \subseteq D'$.
    By \eqref{eqn: annuli without worms acting on discs with worms}, if the operator $Y_{\mathfrak d'}$ exists, then we have
    \[
    Y_{\mathfrak d} = Y_A Y_{\mathfrak d'}.
    \]
    Hence, by replacing $\mathfrak d$ with $\mathfrak d'$, we may assume that $\partial D_1 \cap \partial D$ and $\partial D_2 \cap \partial D$ are intervals.
    Finally, in light of the compatibility of $Y_{\mathfrak d}$ with biholomorphic maps $\varphi \in \mathfrak D(1)$, we may assume without loss of generality that $D = \bbD$.

    So we are reduced to the case of an inclusion $\mathfrak d:D_1 \sqcup D_2 \hookrightarrow \bbD$ of subdiscs $D_1,D_2 \subset \bbD$, under the assumption that $J_1 := \partial D_1 \cap \partial \bbD$ and $J_2 := \partial D_2 \cap \partial \bbD$ are intervals.
    In this case, the subdiscs $D_1,D_2 \subset \bbD$ form a presentation of a star $\Sigma := \bbD \setminus \relint(D_1 \sqcup D_2)$.
    By Lemma~\ref{lem: Omega in X Sigma}, we have $\Omega \in X(\Sigma)$, which is to say that there exists a bounded map 
    \[
    S_{\mathfrak d}:H_0(D_1) \otimes H_0(D_2) \to H_0(\bbD)
    \]
    which satisfies $S_{\mathfrak d}(\Omega_1 \otimes \Omega_2) = \Omega$ and is equivariant for $\cA(J_1) \otimes \cA(J_2)$.
    We will show that $S_{\mathfrak d}$ satisfies \eqref{eqn: characterizing property of Yd}.

    Let $A$ be the annulus $A := \bbD \setminus \mathring D_1$, so that $Y_A:H_0(D_1) \to H_0(\bbD)$.
    First, observe that for $x \in \cA(J_1)$ we have
    \[
    S_{\mathfrak d}(x\Omega_1 \otimes \Omega_2) = xS_{\mathfrak d}(\Omega_1 \otimes \Omega_2) = x\Omega = xY_A\Omega_1 = Y_A(x\Omega_1),
    \]
    where the first equality is $\cA(J_1)$-equivariance of $S_{\mathfrak d}$, and the last one is the $\cA(J_1)$-equivariance of $Y_A$ \cite[Lem. 4.19]{HenriquesTenerWorms}.
    Since $S_{\mathfrak d}$ and $Y_A$ are bounded, the Reeh-Schlieder theorem implies that 
    \begin{equation}\label{eqn: Sd xi1}
    S_{\mathfrak d}(\xi_1 \otimes \Omega_2) = Y_A\xi_1
    \end{equation}
    for all $\xi_1 \in H_0(D_1)$.
    In particular, let $\underline I_1$ and $\underline x_1$ be as in the statement of the theorem, and consider this identity with $\xi_1 = | \underline x_1 \rangle_{D_1}$.
    Let $B := \bbD \setminus \mathring D_2$.
    We then have
    \[
    S_{\mathfrak d}(| \underline x_1 \rangle_{D_1} \otimes \Omega_2) = Y_A| \underline x_1 \rangle_{D_1} = | \underline x_1 \rangle_{\bbD} = B[\underline x_1]\Omega_2,
    \]
    where the last two equalities are given by \eqref{eqn: annuli without worms acting on discs with worms} and \eqref{eqn: annuli with worms acting on disc with worms}, respectively.
    The bounded operators $S_{\mathfrak d}(| \underline x_1 \rangle_{D_1} \otimes -)$ and $B[\underline x_1]$ agree on $\Omega_2$ and are equivariant for $\cA(J_2)$ (the latter by \cite[Lem. 5.5]{HenriquesTenerWorms}). It follows by the Reeh-Schlieder theorem that the operators $S_{\mathfrak d}(| \underline x_1 \rangle_{D_1} \otimes -)$ and $B[\underline x_1]$ are equal.
    If $\underline x_2$ is as in the statement of the Theorem, we thus have
    \[
    S_{\mathfrak d}(| \underline x_1 \rangle_{D_1} \otimes | \underline x_2 \rangle_{D_2})
    =
    B[\underline x_1]| \underline x_2 \rangle_{D_2}
    =
    | \underline x_1 \, \underline x_2 \rangle_{\bbD}.
    \]
    Hence $S_{\mathfrak d}$ satisfies \eqref{eqn: characterizing property of Yd}, completing the proof.
\end{proof}

\begin{rem}\label{rem: Yd when n is 0 or 1}
    Given a disc $D$, we have $D \in \mathfrak D(0)$, and the map $Y_{D}:\bbC \to H_0(D)$ is given by $1 \mapsto \Omega_D$.
    If $D_1 \subseteq D$ is a subdisc, we have $(D_1 \subseteq D) \in \mathfrak D(1)$.
    Comparing the characterizing property \eqref{eqn: characterizing property of Yd} of $Y_{(D_1 \subseteq D)}$ with \eqref{eqn: annuli without worms acting on discs with worms}, we see that $Y_{(D_1 \subseteq D)} = Y_A$ for $A = D \setminus \mathring D_1$.
\end{rem}

We record the operator constructed in Theorem~\ref{thm: Yd exists} as a definition:

\begin{defn}\label{def: Yd}
Let $\cA$ be a conformal net.
    Let $D$ be a disc, and let $\mathfrak d:D_1 \sqcup \cdots \sqcup D_n \to D \in \mathfrak D$.
    Then we denote by
    \[
    Y_{\mathfrak d}: H_0(D_1) \otimes \cdots \otimes H_0(D_n) \to H_0(D)
    \]
    the unique bounded linear operator with the following property: whenever $\underline I_1,\ldots,\underline I_n$ are collections of disjoint extendable intervals in $D_1,\ldots,D_n$, and $\underline x_j \in \cA(\underline I_j)$, we have 
    \begin{equation}\label{eqn: characterizing property of Yd def}
Y_{\mathfrak d} \big( |\underline x_1\rangle_{D_1} \otimes \cdots \otimes |\underline x_n\rangle_{D_n} \big) = |\mathfrak d_*(\underline x_1) \cdots \mathfrak d_*(\underline x_n) \rangle_D.
    \end{equation}
    We call this the genus zero functorial CFT associated to the vacuum sector of the conformal net $\cA$.
\end{defn}

Let $M$ be a complex manifold, and let $D_1, \ldots, D_n$ and $D$ be discs. 
Recall from Definition~\ref{def: multidisc} that a holomorphic family of multidisc embeddings $D_1 \sqcup \cdots \sqcup D_n \to D$ parametrized by $M$ is a holomorphic map $\mathfrak d:M \times (D_1 \sqcup \cdots \sqcup D_n) \to D$ such that the maps $\mathfrak d_m := \mathfrak d(m, -)$ are multidisc embeddings.

\begin{prop}\label{prop: holomorphic dependence}
    Let $M$ be a finite-dimensional complex manifold without boundary, and let $\mathfrak d_m:D_1 \sqcup \cdots \sqcup D_n \to D$ be a holomorphic family of multidisc embeddings parametrized by $m \in M$.
    Then the assignment $m \mapsto Y_{\mathfrak d_m}$ is holomorphic $M \to B\big(H_0(D_1) \otimes \cdots \otimes H_0(D_n), H_0(D)\big)$, with the codomain given the norm topology%
    \footnote{It is a standard fact that a holomorphic map from a finite-dimensional complex manifold without boundary to $B(H)$ is holomorphic for the norm topology if and only if it is holomorphic for the strong or weak operator topologies - see \cite[Lem. 6.1]{HenriquesTenerIntegratingax} for a proof.}%
    .
\end{prop}
\begin{proof}
    The statement is vacuously true in the case $n=0$ (i.e. in the absence of embedded discs).
    We consider first the case $n=1$.
    By fixing identifications $D_1 \cong \bbD \cong D$, we may assume without loss of generality that $D_1=D=\bbD$, in which case $\mathfrak d_m:\bbD \to \bbD$ is a holomorphic family of univalent self-maps of the standard disc.
    Let $A_m$ be the corresponding family of univalent annuli, equipped with their standard lifts $\underline A_m \in \Ann_c$ and regarded as operators on the vacuum Hilbert space $H_0$ (as described in Section~\ref{sec: annuli}).
    The family $\underline A_m$ is holomorphic in the sense of \cite[\S5]{HenriquesTener24ax},  and the corresponding operators depend holomorphically on $M$ by \cite[Thm. 6.4]{HenriquesTenerIntegratingax}%
    \footnote{The claim that the family $\underline A_m$ is holomorphic is not explicit in the references, but an argument is as follows. The family $A_m \in \Ann$ is visibly holomorphic from the definition of the complex diffeology on $\Ann$ \cite[Def. 2.6]{HenriquesTener24ax}. By \cite[Prop. 5.11]{HenriquesTener24ax}, there exist local holomorphic lifts $\underline{{\tilde A}}_m \in \Ann_c$, which differ from the standard lifts $\underline A_m$ by a locally defined function $f:M \to \bbC^\times$. Since the standard lift fixes the vacuum vector, the function $f$ may be recovered as the vacuum expectation of the family $\underline{{\tilde A}}_m$ acting on $H_0$, and is therefore holomorphic. It follows that $\underline A_m$ is a holomorphic family.}%
    .
    By \cite[Lem. 5.4]{HenriquesTenerWorms}, for each $m \in M$ the operators $\underline A_m$ satisfy the defining property \eqref{eqn: characterizing property of Yd def} of $Y_{\mathfrak d_m}$, hence $\underline A_m=Y_{\mathfrak d_m}$. This completes the proof when $n=1$.

    We now consider $n > 1$.
    By working locally in $M$, we may assume without loss of generality that there exist disjoint discs $D_1', \ldots, D_n' \subset D$ such that $\mathfrak d_m(D_j) \subset D_j'$ for all $m \in M$.
    Let $\mathfrak d':D_1' \sqcup \cdots \sqcup D_n' \to D$ be the canonical inclusion, and let $\mathfrak a_{j,m}:D_j \to D_j'$ be the restriction of $\mathfrak d_m$, regarded as a holomorphic family of disc embeddings, so that
    \[
    Y_{\mathfrak d_m} = Y_{\mathfrak d'} \circ \big( Y_{\mathfrak a_{1,m}} \otimes \cdots \otimes Y_{\mathfrak a_{n,m}}\big).
    \]
    The right hand side depends holomorphically on $m$ by the case $n=1$ considered above. This completes the proof.\qedhere
    
\end{proof}

\subsection*{Operators assigned to cobordisms}

By construction, the operators $Y_{\mathfrak d}$ are  natural with respect to isomorphisms in the following sense.
Suppose we have multidisc embeddings $\mathfrak d:D_1 \sqcup \cdots \sqcup D_n \to D$ and $\mathfrak d': D_1' \sqcup \cdots \sqcup D_n' \to D'$.
Suppose there exists a biholomorphic map $\varphi:D \to D'$ such that $\varphi(\mathfrak d(D_j)) = \mathfrak d'(D_j')$ for all $j=1, \ldots, n$.
Let $\mathfrak d_j:D_j \to \mathfrak d(D_j)$ be the biholomorphic map $\mathfrak d_j := \mathfrak d|_{D_j}$, and similarly let $\mathfrak d_j' := \mathfrak d'|_{D_j'}$.
Finally, let $\varphi_j:={\mathfrak{d}_j'}^{-1} \circ \varphi|_{\mathfrak d(D_j)} \circ \mathfrak d_j:D_j \to D_j'$.
Then
\[
Y_{\mathfrak d'} = U_{\varphi}Y_{\mathfrak d}(U_{\varphi_1}^* \otimes \cdots \otimes U_{\varphi_n}^*),
\]
where $U_{\varphi_j}:H_0(D_j) \to H_0(D_j')$ are the unitary operators associated to the biholomorphic maps $\varphi_j$, and similarly $U_\varphi:H_0(D) \to H_0(D')$ is the unitary operator associated to $\varphi$.
We will show below, in Proposition~\ref{prop: stars are well behaved}, that the operators $Y_{\mathfrak d}$ are also natural in a weaker way, with respect to a weaker form of isomorphism.

Let $\mathfrak d, \mathfrak d' \in \mathfrak D(n)$ be multidisc embeddings, as above.
We can then form the cobordisms
\begin{equation}\label{eqn: R from d}
\cR := D \setminus (\mathfrak d(\mathring D_1) \cup \cdots \cup \mathfrak d(\mathring D_n)), \qquad
\cR' := D' \setminus (\mathfrak d'(\mathring D_1') \cup \cdots \cup \mathfrak d'(\mathring D_n')).
\end{equation}
We will show below that the operators $Y_{\mathfrak d}$ are compatible with isomorphisms $\cR \cong \cR'$ (which do not necessarily extend holomorphically to the discs $\mathfrak d(D_j)$), with the caveat that a scalar anomaly arises in the relationship between $Y_{\mathfrak d}$ and $Y_{\mathfrak d'}$.
When $D_i\subset\mathring D$ and $D_i'\subset \mathring D'$, then $\cR$ and $\cR'$ are Riemann surfaces with boundary, and we wish to consider isomorphisms $\cR \cong \cR'$ which are holomorphic in the interior and smooth up to the boundary.
For general inclusions, however, there may be overlap between the discs $D_i$ and $D$, and the cobordism $\cR$ will then fail to be a manifold with boundary.
In this case, the appropriate notion of isomorphism between $\cR$ and $\cR'$ is as follows:

\begin{defn}\label{def: iso of R}
    Let $\mathfrak d:D_1 \sqcup \cdots \sqcup D_n \to D \in \mathfrak D$, $\mathfrak d' : D_1' \sqcup \cdots  \sqcup D_n' \to D' \in \mathfrak D$, $\cR$, and $\cR'$ be as in \eqref{eqn: R from d}.
    An \emph{isomorphism} $\varphi:\cR \to \cR'$ is a homeomorphism which is holomorphic on $\mathring \cR$, maps $\partial D$ to $\partial D'$, maps $\partial (\mathfrak d(D_j))$ to $\partial (\mathfrak d'(D_j'))$, and such that $\varphi_j:={\mathfrak d'}^{-1} \circ \varphi \circ \mathfrak d|_{\partial D_j}:\partial D_j \to \partial D_j'$ and $\varphi_0:=\varphi|_{\partial D}:\partial D \to \partial D'$ are diffeomorphisms.
\end{defn}

Given an isomorphism $\varphi:\cR \to \cR'$, with associated diffeomorphisms $\varphi_0, \varphi_1, \ldots, \varphi_n$ as above, we have implementing unitary operators $U_{\varphi_0}:H_0(D) \to H_0(D')$ and $U_{\varphi_j}:H_0(D_j) \to H_0(D_j')$, for $j=1,\ldots,n$ (which \emph{do not}, in general, map the vacuum vector to the vacuum vector). Such unitaries are well-defined up to a scalar of modulus one.

\begin{prop}\label{prop: stars are well behaved}
     Let $\mathfrak d:D_1 \sqcup \cdots \sqcup D_n \to D \in \mathfrak D$ and $\mathfrak d': D_1' \sqcup \cdots  \sqcup D_n' \to D' \in \mathfrak D$ be multidisc embeddings.
    Let
    \[
    \cR := D \setminus (\mathfrak d(\mathring D_1) \cup \cdots \cup \mathfrak d(\mathring D_n)), \qquad
\cR' := D' \setminus (\mathfrak d'(\mathring D_1') \cup \cdots \cup \mathfrak d'(\mathring D_n')),
\]
    and suppose that $\varphi:\cR \to \cR'$ is an isomorphism (Definition~\ref{def: iso of R}).
     Let $\varphi_j:\partial D_j \to \partial D_j'$  and $\varphi_0:\partial D \to \partial D'$ be the induced diffeomorphisms and let $U_{\varphi_j}$ be implementing unitaries.
     Then
     \[
     Y_{\mathfrak d'} = c \cdot U_{\varphi_0}Y_{\mathfrak d}(U_{\varphi_1}^* \otimes \cdots \otimes U_{\varphi_n}^*)
     \]
     for some scalar  $c \in \bbC^\times$.
\end{prop}
\begin{proof}
    It suffices to consider $\mathfrak d \in \mathfrak D(n)$ for $n=0,1,2$, as these embeddings generate $\mathfrak D$.
    We assume without loss of generality that the discs $D_j$ and $D_j'$ are subdiscs of $D$ and $D'$, respectively, and suppress the embeddings $\mathfrak d$.
    The case $n=0$ is immediate, as in this case $\varphi:D \to D'$ is biholomorphic.
    When $n=1$, as noted in Remark~\ref{rem: Yd when n is 0 or 1},
    we have $Y_{\mathfrak d} = Y_A$ for $A=D\setminus \mathring D_1$. Similarly, $Y_{\mathfrak d'} = Y_{A'}$ for $A'=D' \setminus \mathring D_1'$, and the desired result is a special case of~\eqref{eqn: comparison of unparametrised annuli under iso}.
    
    We now consider the case $n=2$.
    We claim that it suffices to consider the case where $\varphi_j:\partial D_j \to \partial D_j'$ extends to a biholomorphic map $D_j \cong D_j'$ for \emph{either} $j=1$ or $j=2$.
    Indeed, we may construct an intermediate configuration of discs by welding the disc $D_2'$ to the annulus $D \setminus \mathring D_2$ along the boundary identification $\varphi_2$ (see \cite[\S3.2]{HenriquesTener24ax} for a detailed discussion of conformal welding between annuli and discs). This produces a disc $D''$ with subdiscs isomorphic to $D_1$ and $D_2'$, and we let $\cR'' := D'' \setminus (\mathring D_1 \cup  \mathring D_2')$.
    We may then factor $\varphi$ as a composite of isomorphisms $\cR \to \cR'' \to \cR'$, each of which extends to an isomorphism of one of the filling discs.
    It suffices to establish the desired conclusion for each of these intermediate isomorphisms.

    We may thus assume without loss of generality that $\varphi_1$ extends to a biholomorphic map $D_1 \cong D_1'$.
    That is, if we set $A = D \setminus \mathring D_2 = \cR \cup D_1$ and $A' = D' \setminus \mathring D_2' = \cR' \cup D_1'$, then $\varphi$ extends to an isomorphism $A \cong A'$, which we again denote by $\varphi$.
    Let $\underline I_1 \subset D_1$ be a family of disjoint extendable intervals, let $\underline x_1 \in \cA(\underline I_1)$, and let $\xi_1 = | \underline x_1 \rangle_{D_1}$.
    By \eqref{eqn: annuli with worms acting on disc with worms}, we have 
    \begin{equation}\label{eqn: Yd one insertion}
    Y_{\mathfrak d}(\xi_1 \otimes -) = A[\underline x_1],
    \end{equation}
    and similarly if $\underline I_1'$ is a family of disjoint extendable intervals in $D_1'$, $\underline x_1' \in \cA(\underline I_1')$, and $\xi_1' = |\underline x_1'\rangle_{D_1'}$, then
    \begin{equation}\label{eqn: Yd prime one insertion}
    Y_{\mathfrak d'}(\xi_1' \otimes -) = A'[\underline x_1'].
    \end{equation}
    Since $U_{\varphi_1}| \underline x_1 \rangle_{D_1} = |\varphi_*(\underline x_1) \rangle_{D_1'}$, it follows that
    \[
    U_{\varphi_0}^*Y_{\mathfrak d'}(U_{\varphi_1}\xi_1 \otimes U_{\varphi_2}-) = U_{\varphi_0}^*A'[\varphi_*(\underline x_1)]U_{\varphi_2} = c \cdot A[\underline x_1] = c\cdot Y_{\mathfrak d}(\xi_1 \otimes -)
    \]
    for some scalar $c \in \bbC^\times$, where the first equality is \eqref{eqn: Yd prime one insertion}, the second is \eqref{eqn: comparison of unparametrised annuli under iso}, and the last one is \eqref{eqn: Yd one insertion}.
    It follows from the Reeh-Schlieder theorem that
    \[
     U_{\varphi_0}^* Y_{\mathfrak d'}(U_{\varphi_1} \otimes U_{\varphi_2})=c \cdot Y_{\mathfrak d},
    \]
    and the result follows.
\end{proof}

\begin{defn}\label{def: YR}
     Let $\mathfrak d:D_1 \sqcup \cdots \sqcup D_n \to D \in \mathfrak D$ be a multidisc embedding and let
    \[
    \cR := D \setminus (\mathfrak d(\mathring D_1) \cup \cdots \cup \mathfrak d(\mathring D_n))
    \]
    be the associated cobordism.
    Equip the boundary components of $\cR$ with orientation preserving diffeomorphisms $\gamma_j: S^1 \to \partial D_j$ and $\gamma_0:S^1 \to \partial D$, and let $U_{\gamma_j}:H_0 \to H_0(D_j)$ and $U_{\gamma_0}:H_0 \to H_0(D)$ be corresponding implementing unitaries (unique up to unimodular scalar).
    Then we define the operator
    \[
    Y_{\cR}:H_0^{\otimes n} \to H_0
    \]
    by
    \[
    Y_{\cR} := U_{\gamma_0}^* Y_{\mathfrak d}(U_{\gamma_1} \otimes \cdots \otimes U_{\gamma_n}).
    \]
\end{defn}

The operator $Y_\cR$ is well-defined up to a scalar.
More precisely:

\begin{cor}\label{cor: YR well defined}
 Let $\mathfrak d:D_1 \sqcup \cdots \sqcup D_n \to D \in \mathfrak D$ and $\mathfrak d': D_1' \sqcup \cdots  \sqcup D_n' \to D' \in \mathfrak D$ be multidisc embeddings, and 
    let
    \[
    \cR := D \setminus (\mathfrak d(\mathring D_1) \cup \cdots \cup \mathfrak d(\mathring D_n)), \qquad
\cR' := D' \setminus (\mathfrak d'(\mathring D_1') \cup \cdots \cup \mathfrak d'(\mathring D_n'))
\]
be the associated cobordisms.
Suppose that the boundary components of $\cR$ and $\cR'$ are respectively equipped with parametrizations $\gamma_j$ and $\gamma_j'$, and that there exists an isomorphism $\varphi:\cR\to \cR'$ (Definition~\ref{def: iso of R}) such that $\varphi \circ \mathfrak d \circ \gamma_j = \mathfrak d' \circ \gamma_j'$ and $\varphi \circ \gamma_0 = \gamma_0'$. 
Then $$Y_{\cR} = c \cdot Y_{\cR'}$$ for some $c \in \bbC^\times$.
\end{cor}
\begin{proof}
    Let $U_{\varphi_j}:H_0(D_j) \to H_0(D_j')$ be the unitary induced by $\varphi_j:=\varphi|_{\partial D_j}$. 
    Since $\varphi_j \circ \gamma_j = \gamma_j'$, we have
    \(
    Y_\cR = U_{\gamma_0}^* Y_{\mathfrak d}(U_{\gamma_1} \otimes \cdots \otimes U_{\gamma_n}) 
    =c \cdot U_{\gamma_0'}^* Y_{\mathfrak d'}(U_{\gamma_1'} \otimes \cdots \otimes U_{\gamma_n'})
    = c \cdot Y_{\cR'}
    \),
    where the second equality holds by Proposition~\ref{prop: stars are well behaved}.
\end{proof}

Finally, we will observe that the maps $Y_\cR$ are compatible with (de)compositions of cobordisms.

\begin{defn}\label{def: decomposition of R}
Suppose $\mathfrak d_1, \mathfrak d_2 \in \mathfrak D$ and the composition $\mathfrak d:=\mathfrak d_1 \circ_i \mathfrak d_2$ exists (i.e. the $i$th incoming disc of $\mathfrak d_1$ is the same as the outgoing disc of $\mathfrak d_2$).
Let $\cR$, $\cR_1$, $\cR_2$ be the corresponding cobordisms and assume that they are equipped with boundary parametrizations which are compatible in the obvious sense.
In that case, we say that $\cR$ \emph{decomposes} as
\[
\cR = \cR_1 \circ_i \cR_2.
\]
\end{defn}


\begin{cor}\label{cor: YR compatible with decomposition}
    Let $\mathfrak d=\mathfrak d_1 \circ_i \mathfrak d_2$ be as above, and let
    $\cR$, $\cR_1$, and $\cR_2$ be the corresponding cobordisms,
    equipped with compatible boundary parametrizations so that 
    $\cR = \cR_1 \circ_i \cR_2$.
    Then
    \[
    Y_\cR = c \cdot Y_{\cR_1} \circ_i Y_{\cR_2}
    \]
    for some scalar $c \in \bbC^\times$.
\end{cor}
\begin{proof}
    The desired identity is an immediate consequence of the definition of $Y_\cR$ (Definition~\ref{def: YR}) and the fact that $Y_{\mathfrak d} = Y_{\mathfrak d_1} \circ_i Y_{\mathfrak d_2}$ (Theorem~\ref{thm: Yd exists}).

(We note that the particular presentations of $\cR$, $\cR_1$, and $\cR_2$ in terms of $\mathfrak d$, $\mathfrak d_1$, and $\mathfrak d_2$ produce the identity $Y_\cR = Y_{\cR_1} \circ_i Y_{\cR_2}$ without the scalar $c$, but the operators $Y_\cR$ are only well-defined up to a scalar.)
\end{proof}

\section{The trace class condition}

In this section we show that all conformal nets satisfy the trace class condition, which is to say that the operators $r^{L_0}:H_0 \to H_0$ are trace class for all $0 \le r < 1$.
In particular, this implies that conformal nets have finite-dimensional $L_0$-eigenspaces.
We first explain why this should follow from Theorem~\ref{thm: Yd exists}.
Suppose we have a multidisc embedding $\mathfrak d: \bbD \sqcup \bbD \to \bbD$, where $\bbD$ is the standard unit disc.
We then have an associated bounded linear map $Y_{\mathfrak d}:H_0 \otimes H_0 \to H_0$, constructed in Theorem~\ref{thm: Yd exists}.
The vector $Y_{\mathfrak d}^*\Omega \in H_0 \otimes H_0$ corresponds to an antilinear Hilbert-Schmidt operator on $H_0$, and we can show that
this operator is associated to the annulus $\bbC P^1 \setminus  \mathfrak d(\mathring \bbD \sqcup \mathring \bbD)$ obtained by capping off the outgoing boundary of $\cR:=\bbD \setminus \mathfrak d(\mathring \bbD \sqcup \mathring \bbD)$.
After uniformizing this annulus, we see that $r^{L_0}$ is Hilbert-Schmidt for a certain value of $r < 1$, which can be made arbitrarily close to $1$ by bringing the two embedded copies of $\bbD$ closer together.
We may write any operator $r^{L_0}$ as a product of two such operators, and thus they are trace class and not just Hilbert-Schmidt.

We now give the detailed argument:

\begin{thm}\label{thm: rL0 trace class}
Let $\cA$ be a conformal net (not assumed to satisfy the condition that its $L_0$-eigenspaces are finite-dimensional), with vacuum Hilbert space $H_0$. Then the operator $r^{L_0}$ acting on $H_0$ is trace class whenever $0 \le r < 1$.

In particular, the eigenspaces of $L_0$ are finite-dimensional.
\end{thm}
\begin{proof}
It suffices to show that the operators $r^{L_0}$ are Hilbert-Schmidt, as we may factor $r^{L_0} = r_1^{L_0} r_2^{L_0}$, and the product of Hilbert-Schmidt operators is trace class.

Suppose $0 < r,s < 1$ and $z \in \mathring \bbD$ are chosen so that the discs $s \bbD$ and $(z + r \bbD)$ are disjoint and contained in $\mathring \bbD$ (i.e. $s + r < \abs{z} < 1-r$).
Let $f,g:\bbD \to \bbD$ be the univalent maps given by $f(w) = rw + z$ and $g(w) = sw$, and observe that these maps have disjoint image.
Let $\mathfrak d: \bbD \sqcup \bbD \to \bbD$ be the multidisc embedding given by $f$ on the first copy of $\bbD$ and $g$ on the second copy of $\bbD$, and let $Y_{r,s,z}=Y_{\mathfrak d}:H_0 \otimes H_0 \to H_0$ be the corresponding operator.
Let $A_g$ be the annulus corresponding to the univalent map $g$
\[
A_g := \bbD \setminus g(\mathring \bbD) = \bbD \setminus s \mathring \bbD,
\]
and let $\underline A_g \in \Univ$ be the corresponding parametrised annulus with its canonical lift to $\Ann_c$ (Definition~\ref{def: univ}).
Let $x$ and $y$ be worm insertions in $\bbD$, and observe that
\[
Y_{r,s,z}(|x\rangle \otimes |y\rangle) = |f_*(x) g_*(y) \rangle = \underline A_g[f_*(x)] |y \rangle,
\]
where the first equality is the characterizing property of the map $Y_{r,s,z}$ and the second equality is \cite[Lem. 5.4]{HenriquesTenerWorms}.
It follows that
\begin{equation}\label{eqn: Yrsz one insertion one worm}
Y_{r,s,z}(|x\rangle \otimes - ) = \underline A_g[f_*(x)]
\end{equation}
as operators $H_0 \to H_0$, since both operators are bounded and agree on a dense domain (by the Reeh-Schlieder theorem).

Let $I_\pm \subset \partial \bbD$ be the upper (resp. lower) semicircle, and let $\Theta:H_0 \to H_0$ be the modular conjugation associated to $\cA(I_+)$ and the vacuum vector $\Omega$.
Let $\vartheta:\bbC \to \bbC$ be complex conjugation $\vartheta(z) = \bar z$.
As explained in \cite[\S2A]{BartelsDouglasHenriques15} (see also \cite[Lem. 5.12]{HenriquesTenerWorms}), for $x \in \cA(I_+)$ we have 
\begin{equation}\label{eqn: theta on xOmega}
\Theta x\Omega = \vartheta_*(x^*)\Omega,
\end{equation}
where $\vartheta$ is regarded as an orientation-reversing diffeomorphism $I_+ \to I_-$, and $\vartheta_*:\cA(I_+) \to \cA(I_-)^{op}$ is the induced (linear) isomorphism.
Thus for $x \in \cA(I_+)$ we have
\begin{align}
\nonumber \ip{Y_{r,s,z}(\Theta x\Omega \otimes \xi), \Omega} &= 
\ip{Y_{r,s,z}(|\vartheta_*(x^*)\rangle \otimes \xi), \Omega} =
\ip{\underline A_g[(f \circ \vartheta)_*(x^*)] \xi,\Omega}\\ 
&=  \ip{\xi, \underline A_g^\dagger[(f \circ \vartheta)_*(x)]\Omega} \label{eqn: Yrsz vacuum expectation CN}
\end{align}
where the first equality is \eqref{eqn: theta on xOmega}, the second equality is \eqref{eqn: Yrsz one insertion one worm}, and the final equality is \cite[Lem. 5.7]{HenriquesTenerWorms}.
Let $k:\bbD \to \bbD$ be the composite
\[
k = (w \mapsto s \bar w ^{-1}) \circ f \circ \vartheta = (w \mapsto \tfrac{s}{rw + \bar z}).
\]
Using the identification $\underline A_g^\dagger \cong \underline A_g$ induced by the reflection $w \mapsto s\bar w^{-1}$, we have
\begin{equation}\label{eqn: Agdagger CN}
\underline A_g^\dagger[(f \circ \vartheta)_*(x)]\Omega = \underline A_g[k_*(x)]\Omega = |k_*(x)\rangle = \underline A_k x\Omega
\end{equation}
where the first equality is \cite[Cor. 5.13]{HenriquesTenerWorms}, and the last two equalities are \cite[Lem. 5.4]{HenriquesTenerWorms}.
Combining \eqref{eqn: Yrsz vacuum expectation CN} and \eqref{eqn: Agdagger CN}, we have shown that 
\[
\ip{Y_{r,s,z}(\Theta x\Omega \otimes \xi), \Omega} = \ip{\xi, \underline A_k x\Omega}.
\]
Vectors of the form $x\Omega$ are dense by the Reeh-Schlieder theorem, and $\Theta$ is antiunitary.
Since $Y_{r,s,z}$ is a bounded operator, it follows that $\eta \otimes \xi \mapsto \ip{\xi, \underline A_k\eta}$ is a bounded linear functional $\overline H_0 \otimes H_0 \to \bbC$, and hence that $\underline A_k$ is Hilbert-Schmidt.

We may choose an automorphism of $\bbD$ which maps $k(\bbD)$ to a disc $t\bbD$, and thus we may decompose
\[
\underline A_k = U_{\varphi_1} t^{L_0} U_{\varphi_2}
\]
where $\varphi_j \in \Mob$ and $U_{\varphi_j}:H_0 \to H_0$ are the associated unitary operators.
It follows that $t^{L_0}$ is Hilbert-Schmidt.
For fixed values of $z$ and $r$, letting $s$ approach $\abs{z}-r$ produces values of $t$ which converge to $1$.
Hence $t^{L_0}$ is Hilbert-Schmidt for all $t<1$, completing the proof.
\end{proof}

\begin{cor}\label{cor: Yd trace class}
Let $\mathfrak d:D_1 \sqcup \cdots \sqcup D_n \to D$ be a multidisc embedding whose image lies in the interior $\mathring D$. Let $Y_{\mathfrak d}:H_0(D_1) \otimes \cdots \otimes H_0(D_n) \to H_0(D)$ be the bounded operator constructed in Theorem~\ref{thm: Yd exists}.
Then $Y_{\mathfrak d}$ is trace class.
\end{cor}
\begin{proof}
Choose $D'\subset \mathring D$ such that $\mathfrak d$ factors as $D_1 \sqcup \cdots \sqcup D_n \overset{\mathfrak d'}{\to} D' \to D$ and let $A := D \setminus \mathring  D'$. Then $Y_A$ is trace class by the previous theorem, hence so is
$Y_{\mathfrak d} = Y_A \circ Y_{\mathfrak d'}$.
\end{proof}

\appendix

\section{Point insertions}\label{sec: point insertions}
Let $\cA$ be a conformal net, and let $V \subset H_0$ be the subspace of finite-energy vectors (spanned by eigenvectors for $L_0$).
By Theorem~\ref{thm: rL0 trace class}, the eigenspaces for $L_0$ are finite-dimensional.
As a result, by the main theorem of \cite{HenriquesTenerWorms}, there is a unitary vertex operator algebra (VOA) associated with the conformal net $\cA$, whose underlying vector space is $V$.
This unitary VOA was constructed by first introducing vectors in $H_0$ associated to insertions of finite-energy vectors at points inside a disc, analogous to the vectors \eqref{eq: skdjnbksdfksb} associated to worm insertions.
We now briefly describe the relationship between point insertions and worm insertions; for more detail, see \cite[\S6]{HenriquesTenerWorms}.

The space $V$ comes equipped with a positive energy representation $L_n$ of the Virasoro algebra, obtained by differentiating the representation of $\Diff_c(S^1)$ associated to the conformal net.
The operator $L_0$ acts diagonally with nonnegative integer eigenvalues, and as discussed above its eigenspaces are finite-dimensional.
The operators $L_n$ for $n>0$ act locally nilpotently on $V$, and thus the action of the Lie algebra $\Vir_{\ge 0} := \spann\{L_n:n\ge 0\}$ integrates to an action of the pro-algebraic group $\Aut(\bbC[[x]])$ on $V$ (see \cite[\S6.1]{HenriquesTenerWorms} for further discussion).
The group $\Aut(\bbC[[x]])$ may be regarded as the group of changes of formal local coordinate at $0 \in \bbC$.

Given a Riemann surface $\Sigma$ and a point $z\in\Sigma$, define
\begin{equation}\label{eq: def V(z)}
V(z) := V \times_{\Aut(\bbC[[x]])} \mathrm{Isom}(\cO_z(\Sigma),\bbC[[x]]),
\end{equation}
where $\cO_z(\Sigma)$ is the formal completion of the ring of germs of holomorphic functions at $z$, and $\mathrm{Isom}(\cO_z(\Sigma),\bbC[[x]])$ is the set of $\bbC$-algebra isomorphisms from $\cO_z(\Sigma)$ to $\bbC[[x]]$ (i.e., the set of formal coordinates at $z$).
This construction is functorial in the sense that given another Riemann surface $\Sigma'$ and a locally defined holomorphic isomorphism $f:\Sigma\to \Sigma'$ defined on a neighborhood of $z$, we have an induced linear map
\[
f_*:V(z) \to V(f(z)).
\]

For $z=0\in \bbC$, there is a canonical isomorphism between $V$ and $V(0)$ induced by the identity local coordinate, explicitly given by the formula $v\mapsto (v,\mathrm{id}_{\bbC[[x]]})$, and it is convenient to identify these two vector spaces.
More generally, given a point $z\in \bbC$, there is an isomorphism $V(0) \to V(z)$ induced by the local coordinate $\zeta\mapsto \zeta+z$.
Given $v\in V$, we will write $v(z)\in V(z)$ for its image under this isomorphism.

Given a disc $D$, distinct points $z_1, \ldots, z_n \in \mathring D$, and vectors $v_j \in V(z_j)$ for $j=1, \ldots, n$, we constructed in \cite[\S6.2]{HenriquesTenerWorms} a vector
\[
|v_1 \cdots v_n \rangle_D \in H_0(D).
\]
These point insertions encode the state-field correspondence $Y:V \to \End(V)[[z^{\pm 1}]]$ of the vertex algebra  $V$ in the following way.
Given $u_1, \ldots, u_n \in V$, the vectors $|u_1(z_1) \cdots u_n(z_n)\rangle_{\bbD} \in H_0(\bbD)$ depend holomorphically on the choice of distinct points $z_j \in \mathring \bbD$, and the power series expansion of this function in the domain $\abs{z_1} > \cdots > \abs{z_n}$ is given by $Y(u_1,z_1) \cdots Y(u_n,z_n)\Omega$.

The two kinds of insertions, along intervals and along points, are linked in the following way.
First, given a disc $D$ and a point $z \in \mathring D$, there is a natural inclusion $V(z) \hookrightarrow H_0(D)$ \cite[Lem. 6.4]{HenriquesTenerWorms}.
Second, given a vector $v \in V(z)$, there exist intervals $I_1,I_2 \subset \partial D$ and algebra elements $x_j \in \cA(I_j)$ such that $x_1\Omega_D + x_2\Omega_D = v$.
Now given distinct points $z_1, \ldots, z_n \in \mathring D$ and vectors $v_j \in V(z_j)$, choose disjoint discs $D_j \subset \mathring D$, intervals $I_{1,j},I_{2,j} \subset \partial D_j$, and algebra elements $x_{\epsilon,j} \in \cA(I_{\epsilon,j})$ such that $v_j = x_{1,j}\Omega_{D_j} + x_{2,j}\Omega_{D_j}$.
Then, for any such choice of discs $D_j$, intervals $I_{\epsilon,j}$, and algebra elements $x_{\epsilon,j}$, we have
\begin{equation}\label{eqn: def of point insertion in disc}
| v_1 \ldots v_n \rangle_D = \sum_{(\epsilon_1, \ldots, \epsilon_n) \in \{1,2\}^n} | x_{\epsilon_1,1} \cdots x_{\epsilon_n,n} \rangle_D.
\end{equation}
The independence of such expressions on the choices was established in \cite[Lem. 6.7]{HenriquesTenerWorms}.

Finally, as with worms, given a disc $D$, it is sometimes convenient to write $\underline z$ for a tuple $(z_1, \ldots, z_n) \in \mathring D^n$ of distinct points.
If we set $V(\underline z) = V(z_1) \times \cdots \times V(z_n)$, then for $\underline v \in V(\underline z)$, we define
\[
|\underline v\rangle_D := | v_1 \cdots v_n \rangle_D.
\]
Given a collection $\underline v_1, \ldots, \underline v_m$ of such tuples, with all points distinct, we give $|\underline v_1 \cdots \underline v_m \rangle_D$ the obvious meaning of performing all of the insertions associated with the tuples $\underline v_j$.

With the necessary terminology and notation established, we have the following corollaries of Theorem~\ref{thm: Yd exists}:

\begin{cor}\label{cor: action of Yd on point insertions}
    Let $\cA$ be a conformal net with vacuum Hilbert space $H_0$, and let $V \subset H_0$ be the subspace of finite-energy vectors.
    Let $\mathfrak d:D_1 \sqcup \cdots \sqcup D_n \to D$ be a multidisc embedding.
    For $j=1, \ldots, n$, let $\underline z_j$ be a tuple of distinct points in $\mathring D_j$, and let $\underline v_j \in V(\underline z_j)$ be a tuple of vectors.
    Then the bounded operator $Y_{\mathfrak d}: H_0(D_1) \otimes \cdots \otimes H_0(D_n) \to H_0(D)$ constructed in Theorem~\ref{thm: Yd exists} satisfies
    \begin{equation*}\label{eqn: characterizing property of YR points}
    Y_{\mathfrak d}(|\underline v_1\rangle_{D_1} \otimes \cdots \otimes |\underline v_n\rangle_{D_n}) = | \mathfrak d_*(\underline v_1) \cdots \mathfrak d_*(\underline v_n) \rangle_D.
    \end{equation*}
\end{cor}
\begin{proof}
    Each vector $|\underline v_j\rangle_{D_j}$ is, by definition \eqref{eqn: def of point insertion in disc}, a sum of $2^{m_j}$ insertions along intervals, where $m_j$ is the length of the tuple $\underline v_j$.
    By the characterizing property \eqref{eqn: characterizing property of Yd} of $Y_{\mathfrak d}$, the output $Y_{\mathfrak d}(|\underline v_1\rangle_{D_1} \otimes \cdots \otimes |\underline v_n\rangle_{D_n})$ is therefore equal to a sum of $2^{m_1 + \cdots m_n}$ insertions along intervals.
    By definition \eqref{eqn: def of point insertion in disc}, this sum is equal to $| \mathfrak d_*(\underline v_1) \cdots \mathfrak d_*(\underline v_n) \rangle_D$.
\end{proof}

\begin{cor}\label{cor: regularized vertex operators are trace class}
Let $\cA$ be a conformal net with vacuum Hilbert space $H_0$, and let $V \subset H_0$ be the subspace of finite-energy vectors equipped with the structure of a unitary VOA as in \cite{HenriquesTenerWorms}.
Let $Y:V \to \End(V)[[z^{\pm 1}]]$ be the state-field correspondence of this VOA.
Let $0 < r,s < 1$ and $z \in \mathring \bbD$ be such that the discs $s \bbD$ and $(z + r \bbD)$ are disjoint and contained in $\mathring \bbD$ (i.e. $s + r < \abs{z} < 1-r$).
Let $f:\bbD \to \bbD$ and $g:\bbD \to \bbD$ be the univalent maps
    \[
    f(w) = rw + z, \qquad g(w) = sw,
    \]
and let $\mathfrak d: \bbD \sqcup \bbD \to \bbD$ be the multidisc embedding which acts by $f$ on the first copy of $\bbD$ and by $g$ on the second.
Let $Y_{\mathfrak d}:H_0 \otimes H_0 \to H_0$ be the associated operator constructed in Theorem~\ref{thm: Yd exists}.
Then for $v,u \in V$ we have
\[
Y_{\mathfrak d}(v \otimes u) = Y(r^{L_0}v,z)s^{L_0}u.
\]
In particular, the expression on the right-hand side defines a map $V \otimes V \to H_0$ which extends to a trace class map $H_0 \otimes H_0 \to H_0$. 
\end{cor}
\begin{proof}
    From the definition of the VOA structure on $V$ and Corollary~\ref{cor: action of Yd on point insertions}, we have
    \[
    Y(r^{L_0}v,z)s^{L_0}u = | f_*(v) g_*(u) \rangle_\bbD = Y_{\mathfrak d}(v \otimes u).
    \]
    This map is trace class by Corollary~\ref{cor: Yd trace class}.
\end{proof}


\def\lfhook#1{\setbox0=\hbox{#1}{\ooalign{\hidewidth
  \lower1.5ex\hbox{'}\hidewidth\crcr\unhbox0}}}

\end{document}